\documentclass{amsart}
\usepackage{amscd}
\usepackage{amssymb}
\usepackage[all, knot]{xy}
\usepackage{bbm}
\usepackage[top=1.2in, bottom=1.2in, left=1.2in, right=1.2in]{geometry}
\xyoption{all}
\xyoption{arc}
\usepackage{hyperref}
\usepackage[mathscr]{euscript}

\def\hen{{\mathrm{\hen}}}
\def\len{\operatorname{length}}
\def\V{V^{>-1}}
\def\Res{\operatorname{Res}}

\def\Gr{\operatorname{Gr}}

\def\hQ{\widehat{Q}}
\def\hB{\widehat{B}}

\def\del{\partial}

\def\cube#1#2#3#4#5#6#7#8{
& #5 \ar[rr] \ar[dl] \ar@{-}[d] && #6 \ar[dd] \ar[dl] \\
#1 \ar[rr] \ar[dd]  & \ar[d] & #2 \ar[dd] \\
& #7 \ar@{-}[r] \ar[dl] & \ar[r] & #8 \ar[dl] \\
#3 \ar[rr] && #4 \\
}

\def\smsh{\wedge}

\def\etale{\'etale~}

\def\image{\operatorname{im}}
\def\im{\image}
\def\ker{\operatorname{ker}}

\def\cE{\mathcal E}

\def\cC{\mathcal C}

\def\cG{\mathcal G}
\def\cO{\mathcal O}

\def\coker{\operatorname{coker}}

\def\dm{\operatorname{dim}}

\def\Sing{\operatorname{Sing}}
\def\Proj{\operatorname{Proj}}

\def\Ext{\operatorname{Ext}}

\def\sExt{\sExt}

\def\Tor{\operatorname{Tor}}

\def\End{\operatorname{End}}

\def\Spec{\operatorname{Spec}}

\def\bu{\bullet}
\def\an{{\mathrm{an}}}

\def\into{\hookrightarrow}
\def\onto{\twoheadrightarrow}

\def\o#1{{\overline{#1}}}

\def\hOmega{\widehat{\Omega}}

\def\a{\alpha}
\def\b{\beta}
\def\g{\gamma}
\def\e{\epsilon}

\def\rk{\operatorname{rank}}

\DeclareMathOperator*{\colim}{colim}

\newcommand{\Q}{\mathbb{Q}}

\newcommand{\bP}{\mathbb{P}}

\newcommand{\A}{\mathbb{A}}

\newcommand{\G}{\mathbb{G}}

\newcommand{\R}{{\mathbf{R}}}

\newcommand{\C}{\mathbb{C}}
\newcommand{\Z}{\mathbb{Z}}

\newcommand{\fm}{{\mathfrak m}}

\numberwithin{equation}{section}

\theoremstyle{plain} 
\newtheorem{thm}[equation]{Theorem}
\newtheorem{conj}[equation]{Conjecture}
\newtheorem{thm-conj}[equation]{Theorem-Conjecture}
\newtheorem{defn-conj}[equation]{Definition-Conjecture}

\newtheorem*{introthm*}{Theorem}

\newtheorem{cor}[equation]{Corollary}
\newtheorem{lem}[equation]{Lemma}
\newtheorem{prop}[equation]{Proposition}

\theoremstyle{definition}
\newtheorem{defn}[equation]{Definition}

\newtheorem{ex}[equation]{Example}

\theoremstyle{remark}
\newtheorem{rem}[equation]{Remark}

\newtheorem{assumptions}[equation]{Assumptions}

\def\Perf{\operatorname{Perf}}

\newcommand{\Hom}{\operatorname{Hom}}

\newcommand{\xra}[1]{\xrightarrow{#1}}

\newcommand{\id}{\operatorname{id}}

\def\tr{\operatorname{tr}}

\def\and{ \text{ and } }

\def\can{\mathrm{can}}

\def\op{{\mathrm{op}}}

\def\G{\Gamma}

\def\o{\omega}

\def\on{\operatorname}

\def\ch{\on{ch}}
\def\fC{\EuScript{C}}
\def\fD{\EuScript{D}}

\def\hen{\on{hen}}
\def\th{\on{th}}
\def\p{{\frac{n+1}{2}}}

\def\idem{\on{idem}}

\begin{document}
\begin{abstract}
We prove the non-commutative analogue of Grothendieck's Standard Conjecture $D$ for the dg-category of matrix factorizations of an isolated hypersurface singularity in
characteristic $0$. Along the way, we show the Euler pairing for such dg-categories of matrix factorizations is positive semi-definite. 
\end{abstract}
\title{Standard conjecture $D$ for matrix factorizations}

\author{Michael K. Brown}
\address{Department of Mathematics, University of Wisconsin-Madison, WI 53706-1388, USA}
\email{mkbrown5@wisc.edu}

\author{Mark E. Walker}
\address{Department of Mathematics, University of Nebraska-Lincoln, NE 68588-0130, USA}
\email{mark.walker@unl.edu}

\thanks{MB and MW gratefully acknowledge support from the National Science Foundation (NSF award DMS-1502553) and
 the Simons Foundation (grant  \#318705), respectively.} 
\maketitle

\tableofcontents

\section{Introduction}

Let $k$ be a field. Grothendieck's Standard Conjecture $D$ predicts that numerical equivalence and homological equivalence coincide for cycles on 
a smooth, projective variety $X$ over $k$. 
Marcolli and Tabuada (\cite{MT}, \cite{tabuada2})
have formulated a non-commutative generalization of this conjecture, referred to as Conjecture $D_{nc}$, which  
predicts that numerical equivalence and homological equivalence coincide for a smooth and proper dg-category $\fC$ over a field $k$. 
In this paper, we prove
that Conjecture $D_{nc}$ (more precisely, its $\Z/2$-graded analogue) holds for the differential $\Z/2$-graded
category of matrix factorizations associated to an isolated hypersurface singularity over a field of characteristic $0$. 

Before stating our results precisely, 
we give some background.
\subsection{Grothendieck's Standard Conjecture $D$} 
Let $X$ be a smooth, projective $k$-variety. We write
$Z^j(X)$ for the group of codimension $j$ algebraic cycles on $X$; by definition, it is the
free abelian group on the set of integral subvarieties of $X$ having
codimension $j$. Let $H^*(-)$ be any Weil cohomology theory for smooth projective $k$-varieties; as a concrete example, the reader may take $k =\C$ and $H^*( - )$ to be singular
cohomology. There is an associated cycle class map
$$
cy: Z^j(X) \to H^{2j}(X),
$$
and two algebraic cycles on $X$ are {\em homologically equivalent} if their images in $H^*(X)$ 
under this map  coincide.  Let $\langle -,- \rangle$ denote the intersection pairing for cycles, determined by $\langle W, Z \rangle = \deg(W \cap Z)$ for integral
subvarieties $Z$, $W$ of $X$ meeting properly at a finite number of points. Two cycles $\a, \b$ are {\em numerically equivalent} if $\langle \a,- \rangle = \langle \b,- \rangle$. 
Since the intersection pairing is induced by a pairing on $H^*$ under the cycle class map, it is immediate that whenever two cycles are homologically equivalent they are
numerically equivalent. Conjecture $D$ predicts the converse holds:

\begin{conj}[Grothendieck's Standard Conjecture $D$]
\label{D}
For any field $k$, Weil cohomology theory $H^*$, and smooth, projective variety $X$ over $k$, if two cycles on $X$ are numerically equivalent then they
  are homologically equivalent.
\end{conj}

Conjecture $D$ remains open in general. It is known to hold, for instance, when $X$ is a complete intersection 
(we sketch the proof for complex hypersurfaces in Section \ref{maintheorem}), and, by work of Lieberman \cite{lieberman}, it holds both when $X$ is an abelian variety and
when $\dim X \le 4$.
\subsection{Noncommutative analogue}
\label{introNC}

Assume now that $\on{char}(k) = 0$.
 Let $\fC$ be a differential $\Z$-graded category over $k$, i.e.~a category enriched over $\Z$-graded complexes of $k$-vector spaces.  We say  $\fC$ is 
\begin{itemize}
\item \emph{smooth} if the $\fC^{\op} \otimes \fC$-module determined by $\fC$ is perfect, and
\item \emph{proper} if the total homology of $\Hom_{\fC}(\a,\b)$ is finite dimensional as a $k$-vector space for all objects $\a$ and
$\b$.
\end{itemize}
Assume that $\fC$ is smooth and proper. To formulate Conjecture $D_{nc}$ for $\fC$, one needs analogues of 
\begin{itemize}
\item a Weil cohomology theory,
\item algebraic cycles,
\item the cycle class map, and
\item the intersection
pairing.
\end{itemize}
These are given by
\begin{itemize}
\item the periodic cyclic homology of $\fC$, written as  $HP_*(\fC)$,
\item classes in the rational Grothendieck group $K_0(\fC)_\Q$,
\item the Chern character map $ch_{HP}: K_0(\fC)_\Q \to HP_0(\fC)$, and
\item the Euler pairing $\chi(-,-)_\fC$.
\end{itemize}

See Section \ref{background} for the definition of the Grothendieck group of a dg-category, and see, for instance, Sections 3 and 4 of \cite{BW} for the definitions of $HP_*(\fC)$ and the Chern character map $ch_{HP}$, respectively.
The Euler pairing is defined on a pair of objects $P, P' \in \Perf(\fC)$
to be
$$
\chi(P, P')_\fC := \sum (-1)^i \dm_k H^i \Hom_{\Perf(\fC)}(P,P').
$$
Since $\fC$ is smooth and proper, $\Perf(\fC)$ is as well \cite[Prop. 13]{toen}, and thus the pairing is well-defined. 
(One really just needs $\Perf(\fC)$ to be proper for the Euler pairing to be well-defined.)

Classes $\a, \b$ in $K_0(\fC)_\Q$ are said to be {\em homologically equivalent} if $ch_{HP}(\a) = ch_{HP}(\b)$, and they are said to be {\em numerically equivalent} if the functions
$$
\chi(\a, -)_\fC,  \chi(\b, -)_\fC : K_0(\fC)_\Q \to \Q
$$
coincide, or, equivalently, if $\chi(\a - \b, -)_\fC$ is the zero function. 

We may now state Conjecture $D_{nc}$ in characteristic $0$: 
\begin{conj}[\cite{MT}]
\label{DNC}
If $\fC$ is a smooth and proper differential $\Z$-graded category over a field $k$ of characteristic $0$,
homological and numerical equivalence coincide for $\fC$. 
\end{conj}

\begin{rem}
By a theorem of Marcolli-Tabuada \cite[Theorem 1.1]{MTK}, the notion of numerical equivalence described above coincides with that of \cite[Section 3.2]{MT}, and it follows directly from the definition of $ch_{HP}$ that the above notion of homological equivalence agrees with that of \cite[Section 10]{MT}.
\end{rem}

\begin{rem}
A positive characteristic version of Conjecture $D_{nc}$ is posed in \cite{tabuada2}, where the role of periodic cyclic homology is played by topological periodic cyclic homology.
\end{rem}

It follows from work of Shklyarov \cite[Theorems 2 and 3]{shklyHRR} that the Euler pairing factors through the map $ch_{HP}$. Therefore, just as in the classical setting, classes homologically equivalent to $0$ are numerically
equivalent to $0$; that is, the content of Conjecture $D_{nc}$ is: 
\begin{quote}
Given a class $\a \in K_0(\cC)_\Q$, if $\chi(\a, \b)_\fC = 0$ for all $\b \in K_0(\cC)_\Q$, then $ch_{HP}(\a) = 0$.
\end{quote}

Conjectures \ref{D} and \ref{DNC} are related by a theorem 
of Tabuada \cite[Theorem 1.1]{tabuada}, which states that, for a smooth, projective variety $X$ over a field of characteristic $0$,
Conjecture $D$ holds for $X$ if and only if 
Conjecture $D_{nc}$ holds for the dg-category $\Perf(X)$ of perfect complexes on $X$. 

One may also state an analogue of Conjecture $D_{nc}$ for differential $\Z/2$-graded categories $\fC$, i.e. categories enriched over $\Z/2$-graded complexes of $k$-vector spaces. The notions of smooth and
proper generalize in an evident manner to this setting, as do the constructions $HP_*(-)$, $K_0(-)_\Q$, $ch_{HP}$, $\chi(-,-)_{\fC}$ and the notions of numerical and homological equivalence. One shows that homological equivalence implies numerical equivalence by adapting \cite[Theorems 2 and 3]{shklyHRR} to the $\Z/2$-graded setting.

\begin{ex}
If $Q$ is a (non-graded) commutative $k$-algebra, and $f \in Q$ is any element, matrix factorizations of $f$ form a $k$-linear differential $\Z/2$-graded category, written
$mf(Q,f)$; see   \cite[Definition 2.1]{dyckerhoff}. Section \ref{mf} below contains additional background on matrix factorizations. 
By a theorem of Preygel (\cite[Theorem 8.1.1(iii)]{preygel}), if $Q$ is smooth over $k$, and $\Sing(Q/f)$ is zero dimensional, $mf(Q,f)$ is smooth and proper. 
\end{ex}

\subsection{Main theorem}
\label{maintheorem}

We may now state our main result. 

\begin{thm} \label{introthm} 
Let $k$ be a field of characteristic $0$, 
$Q$ a smooth $k$-algebra, 
and $f \in Q$ a non-zero-divisor such that the singular locus of $\Spec(Q/f)$ is a finite set of points. 
Then, for $\a \in K_0(mf(Q,f))$, we have
\begin{itemize}
\item $\chi(\a, \a)_{mf(Q,f)} \geq  0$, and
\item $\chi(\a, \a)_{mf(Q,f)} = 0$ if and only if $ch_{HP}(\a) = 0$. 
\end{itemize}
In particular, the analogue of Conjecture
$D_{nc}$ for smooth and proper differential $\Z/2$-graded categories holds for $mf(Q,f)$.  
\end{thm}

To give an idea of how the proof goes, we begin by reviewing the proof of Conjecture $D$ for a smooth, projective complex hypersurface $Y \subseteq
\bP^{n}_\C$.  This is more than a mere analogy: if $Q = k[x_0, \dots, x_n]$, and $f$ is a homogeneous polynomial, the Euler pairing for $mf(Q,f)$ is
explicitly related to the classical intersection pairing on the smooth projective hypersurface $Y = \Proj(Q/f)$; see \cite{MPSW} for details.

Conjecture $D$ amounts to the following assertion:
\begin{quote}
Given a cycle $\a \in Z^j(Y)$, if $\langle cy(\a), cy(\b) \rangle = 0$ for all cycles $\b \in
Z^{n-1-j}(Y)$, then $cy(\a) = 0$.
\end{quote}
Here, $\langle -,- \rangle$ denotes the pairing on $H^*(Y,\Q)$ given as the composition 
$$
H^*(Y,\Q) \otimes H^*(Y,\Q) \xra{\cup} H^*(Y,\Q) \onto H^{2(n-1)}(Y,\Q) \cong \Q.
$$

Let $h \in H^2(Y; \Q)$ be the cohomology class of a generic hyperplane section of $Y$. 
Then 
\begin{itemize}
\item $H^{2j}(Y; \Q) = \Q \cdot h^j \cong \Q$ whenever $2j \ne n-1$, and
\item $H^{2j+1}(Y; \Q) = 0$ whenever $2j+1 \ne n-1$.
\end{itemize}

Suppose $\a \in Z^j(Y)$ satisfies $\langle cy(\a), cy(\b) \rangle = 0$ for all cycles $\b \in Z^{n-1-j}(Y)$. 
If $cy(\a) = q h^j$ for $q \in \Q$, then $\langle cy(\a), h^{n-1-j}\rangle = q \deg(h^{n-1}) = q \deg(Y)$. Since $h$ is algebraic and $\deg(Y) >0$, we conclude
$cy(\a) = 0$. In particular, we may assume $n$ is odd and $j = \frac{n-1}{2}$. 
Moreover, the map $h^{j-1} \cap - : H^{2j+2}(Y; \Q) \to H^{2n-2}(Y; \Q)$ is an isomorphism, and hence $cy(\a) \cap h = 0$. That
is, $cy(\a)$ belongs to $PH^{n-1}(Y)$, where $PH^*$ denotes the primitive part of the cohomology of a smooth projective variety.

Finally, classical Hodge theory gives that, for any smooth projective complex variety $X$, 
the intersection pairing is either positive or negative definite (depending on the parity of $\dm(X)$) 
on
$$
\im(cy: Z^*(X) \to H^{2*}(X; \Q)) \cap P H^{2*}(X; \Q).
$$
Since $\langle cy(\a), cy(\a) \rangle = 0$, it follows that $cy(\a) = 0$. 

Our proof of Theorem \ref{introthm} parallels the above proof.  We first reduce to the case where $k = \C$, $U := \Spec(Q)$ is a Zariski open
neighborhood of the origin in $\A^{n+1}_\C$, and the only singularity of $f: U \to \A^1_\C$ is at the origin; this is the content of Section \ref{reduction}. In this situation, we have the associated (universal) Milnor fiber
$X_\infty$, whose singular cohomology $H^n(X_\infty; \C)$ in degree $n$ is a direct sum of polarized mixed Hodge structures.
The role of $H^*(Y; \Q)$ in the proof sketched above is played, in our proof of Theorem \ref{introthm}, by $H^n(X_\infty; \C)_1$, the summand of $H^n(X_\infty; \C)$ on which the operator $(M - \id)$ acts nilpotently,  where $M$ is the monodromy operator. We recall the necessary background concerning the Milnor fiber in Section \ref{milnorbackground}.

We prove Theorem \ref{introthm} by establishing the following facts:
\begin{enumerate}
\item There is a map $ch_{X_\infty}: K_0(mf(Q, f)) \to H^n(X_\infty; \Q)_1$
such that the polarizing form $S$ on $H^n(X_\infty; \Q)_1$ is positive definite on the image of $ch_{X_\infty}$ (Proposition \ref{prop915} and Corollary \ref{cor913}).

\item The pairings $S(ch_{X_\infty}(-), ch_{X_\infty}(-))$ and $\chi(-, -)_{mf(Q,f)}$ coincide (Theorem \ref{BvS}).

\item
$ch_{X_\infty}(\a) = 0$ if and only if $ch_{HP}(\a) = 0$ (Theorem \ref{thm822}).
\end{enumerate}

\begin{rem}
\label{bvsrem}
Step (2) in the above sketch of our proof of Theorem \ref{introthm} may be thought of as an analogue of Polishchuk-Vaintrob's Hirzebruch-Riemann-Roch formula for matrix factorizations (see (\ref{PVresult}) below). It was inspired by a similar result of Buchweitz-van Straten (\cite[Main Theorem (ii), p.~245]{BvS}) which compares Hochster's theta pairing (defined below in Section \ref{application}) to the linking pairing associated to a complex hypersurface singularity.
\end{rem}

\begin{rem}
\label{referee}
Theorem \ref{introthm} implies that the canonical pairing on $HP_0(mf(Q, f))$ (see \cite{shklyresidue} for the definition of this pairing) is positive definite on the image of $ch_{HP}$. On the other hand, the intersection form on the cohomology of a projective hypersurface of even dimension is positive definite only when restricted to primitive cohomology. We explain the reason for this discrepancy via the following example. 

Suppose, in the setting of Theorem \ref{introthm}, that $Q = \C[x_0, \dots, x_n]$, where $n$ is odd, and $f \in Q$ is homogeneous of degree $n$. Then $X = \Proj(Q/(f))$ is smooth and Calabi-Yau. By a famous result of Orlov (\cite[Corollary 2.15]{orlovCY}), it follows that there is a quasi-equivalence of smooth and proper differential $\Z$-graded categories
$$
\Perf(X) \xra{\simeq} mf^{\on{gr}}(Q, f),
$$
where $mf^{\on{gr}}(Q, f)$ denotes the dg-category of $\Z$-graded matrix factorizations of $f$. We therefore have an induced quasi-equivalence of smooth and proper differential $\Z/2$-graded categories
$$
\Perf(X)^{\Z/2} \xra{\simeq} mf^{\on{gr}}(Q, f)^{\Z/2},
$$
where $( - )^{\Z/2}$ denotes the $\Z/2$-folding of a differential $\Z$-graded category. There is an obvious functor
\begin{equation}
\label{forgetgrading}
mf^{\on{gr}}(Q, f)^{\Z/2} \to mf(Q, f)
\end{equation}
of differential $\Z/2$-graded categories given by forgetting the grading. $H^*(X ; \C)$ and $HP_0(mf(Q, f))$ are explicitly related by the composition
$$
HP_0(\Perf(X)^{\Z/2}) \xra{\cong} HP_0(mf^{\on{gr}}(Q, f)^{\Z/2}) \to HP_0(mf(Q, f));
$$
the fact that (\ref{forgetgrading}) is not typically an equivalence accounts for the discrepancy. 

\end{rem}

\subsection{Application to a conjecture in commutative algebra}
\label{application}
As an application of the positive semi-definiteness statement in Theorem \ref{introthm}, we make progress
on a conjecture in commutative algebra concerning ``Hochster's  theta pairing", whose definition we now recall. Let $Q$ and $f$ be as in Theorem \ref{introthm},
and set $R := Q/(f)$. Any finitely generated $R$-module has an eventually 2-periodic projective resolution, and, because of the assumption on the singular locus of $R$,
sufficiently high Tor's between finitely generated $R$-modules are of finite length. 
\emph{Hochster's theta pairing} is the map
$\theta : G_0(R) \times G_0(R) \to \Z$ given by
$$
([M], [N]) \mapsto \len_R \Tor_{2i}^R(M, N) - \len_R \Tor_{2i + 1}^R(M, N), \text{ } i \gg 0.
$$

We prove the following in Section \ref{hochstersection}:
\begin{thm}
\label{hochster}
If $k$, $Q$ and $f$ are as in Theorem \ref{introthm}, and $Q$ is equi-dimensional, then  $(-1)^{\frac{\dim Q}{2}} \theta(-,-)$ is positive semi-definite.
\end{thm}

When $\dm Q$ is odd, the conclusion of this theorem  is interpreted as meaning $\theta = 0$, which was proven independently by Buchweitz-van Straten and
Polishchuk-Vaintrob (\cite[Main Theorem (i), p.~245]{BvS}, \cite[Remark 4.1.5]{PV}). 
The case when $\dm Q$ is even settles a conjecture of Moore-Piepmeyer-Spiroff-Walker \cite[Conjecture 3.6]{MPSW} in characteristic $0$.
This was previously known in the case where $R$ is a \emph{graded} hypersurface (loc. cit. Theorem 3.4).

The proof of Theorem \ref{hochster} uses techniques in algebraic and topological $K$-theory. In more detail: we recall that $G_0(R) \otimes \Q$ is isomorphic to $\bigoplus_{i} A_i(R) \otimes \Q$, where $A_i(R)$ denotes the group of dimension $i$ cycles in $\Spec(R)$ modulo rational equivalence. Assume $\dm(Q)$ is even. The key step in the argument is a proof of a special case of a conjecture of Dao-Kurano (\cite[Conjecture 3.4]{DK}) which predicts that $\theta( \a,  - ) : G_0(R) \to \Z$ is the zero map for any $\a$ corresponding to a class in  $A_i(R) \otimes \Q$ such that $i \ne \frac{\dim(Q)}{2}$. This is the content of Theorem \ref{adams}, whose proof uses techniques from topological $K$-theory developed in \cite{BvS} and expanded on in \cite{brown}. 

As an immediate corollary of Theorem \ref{adams}, we conclude a special case of another conjecture of Dao-Kurano (\cite[Conjecture 3.1(4)]{DK}) which predicts that, if $M$ and $M'$ are MCM $R$-modules, then $\theta(M^*, M') = -(-1)^{\frac{\dim(Q)}{2}}\theta(M, M')$, where $M^*$ denotes the $R$-linear dual of $M$; this is Corollary \ref{theta}.
\vskip\baselineskip
\noindent {\bf Acknowledgements.} We thank the referee for his or her careful reading and helpful suggestions. 

\section{Background}
\label{background}

Throughout the paper, $k$ denotes a  field of characteristic $0$. 

\subsection{The Grothendieck group of a dg-category}
\label{K0}

For the rest of this paper, ``dg-category'' means ``$k$-linear differential $\Z/2$-graded category", unless otherwise specified. We recall here a bit of background concerning dg-categories, and we refer the reader to \cite{toen} for a comprehensive introduction.
\begin{itemize}

\item The \emph{homotopy category} of a dg-category $\fC$ is the category with the same objects as $\fC$ and morphisms given by the $k$-vector spaces
  $H^0 \Hom_{\fC}(-,-)$. We write $[\fC]$ for  the homotopy category of $\fC$.

\item Given a dg-category $\fC$, let $\Perf(\fC)$ denote the dg-category of perfect right $\fC$-modules, i.e.~the triangulated hull of $\fC$, in the sense of \cite{toen}.
\item A dg functor $\fC \to \fD$ is a \emph{quasi-equivalence} if 
\begin{itemize}
\item the maps on morphism complexes are quasi-isomorphisms, and
\item the induced map $[\fC] \to [\fD]$ on homotopy categories is essentially surjective.
\end{itemize}
\item A dg functor $F: \fC \to \fD$
is a \emph{Morita equivalence} if the induced map
$$
F^* : \Perf(\fD) \to \Perf(\fC)
$$
on triangulated hulls is a quasi-equivalence. 
\end{itemize}

\begin{rem} 
\label{pretri}
When $\fC$ and $\fD$ are pretriangulated, as defined in \cite[Section 3]{BK}, a dg functor $F: \fC \to \fD$ 
is a quasi-equivalence (resp., a Morita equivalence) if and only if the induced functor  $[\fC] \to [\fD]$ (resp., $[\fC]^{\idem} \to [\fD]^{\idem}$)
is an equivalence. Here, the superscript ``$\idem$" indicates the idempotent completion of a triangulated category; see \cite{BS} for details.
\end{rem}

For any triangulated category  $T$, we write $K_0^{\Delta}(T)$ for its Grothendieck group, i.e.~the free abelian group on isomorphism classes of
objects of $T$ modulo relations given by exact triangles.
The \emph{Grothendieck group} $K_0(\fC)$ of a dg-category $\fC$ 
is defined to be $K_0^\Delta([\Perf(\fC)])$.
If $\fC$ is pre-triangulated, there is a canonical isomorphism $K_0(\fC) \cong K_0^{\Delta}([\fC]^{\on{idem}})$. (Note that the
canonical map $K_0^\Delta([\fC]) \to K_0^{\Delta}([\fC]^{\on{idem}})$ need not be an isomorphism.)

Recall from Section \ref{introNC} that, when $\fC$ is a proper dg category, $K_0(\fC)$ is equipped with an Euler pairing $\chi$; we only defined $\chi$ in the
setting of differential $\Z$-graded categories, but, as noted in Section \ref{introNC}, there is an evident analogue for differential $\Z/2$-graded
categories. When $\fC$ is pretriangulated, 
the canonical isomorphism
$$
K_0(\fC) \cong K_0^{\Delta}([\fC]^{\on{idem}})
$$
identifies $\chi$ with  the pairing on $K_0^{\Delta}([\fC]^{\on{idem}})$ given as follows: 
Recall from \cite{BS} that the objects of $[\fC]^{\on{idem}}$ are
pairs $(P, e)$ with $P$ an object of $\fC$ and $e \in \End_{[\fC]}(P, P)$ an idempotent. A morphism 
from $(P, e)$ to $(P', e')$ is a morphism $\a: P \to P'$ in $[\fC]$ such that $\a \circ e = e' \circ \alpha = \alpha$.   
The pairing is given by 
\begin{align*}
\langle [P, e], [P', e'] \rangle 
& =  \dim_k\Hom_{[\fC]^{\on{idem}}}((P, e), (P', e')) - \dim_k \Hom_{[\fC]^{\on{idem}}}((P, e), (P'[1], e'[1])),
\end{align*}
where $[1]$ denotes the translation functor for the triangulated category $[\fC]$. 

\subsection{Matrix factorizations}
\label{mf}

Let $Q$ be a regular $k$-algebra and $f \in Q$ a non-zero-divisor. The dg-category $mf(Q,f)$ of matrix factorizations of $f$ is defined as follows: 
\begin{itemize}
\item An object is a pair $(P,d)$ (usually written as just $P$), where $P$ is a finitely generated $\Z/2$-graded  projective
$Q$-module written $P = P_1 \oplus P_0$, and $d$ is an odd degree $Q$-linear endomorphism such that $d^2$ is multiplication by $f$.
\item For any two objects $P = (P,d)$ and  $P' = (P', d')$, $\Hom_{mf}(P,P')$
is the $\Z/2$-graded complex of finitely generated projective $Q$-modules $\Hom_Q(P,P')$ with differential $\del$ given by
$$
\del(\alpha) = d' \circ \a - (-1)^{|\a|}
\a \circ d.
$$
\item The composition rule and identities are the obvious ones.
\end{itemize}  

Set $R := Q/(f)$, and let $D^b(R)$ denote the dg quotient of the differential $\Z$-graded category of bounded chain complexes of $R$-modules by the subcategory spanned by acyclic complexes. $D^b(R)$ is a dg enhancement of the bounded derived category of $R$; this enhancement is unique by \cite[Theorem 8.13]{LO}).
Let $D^b(R)/\Perf(R)$ denote the dg quotient of $D^b(R)$ by the subcategory spanned by perfect complexes. By a theorem of Buchweitz (\cite{buchweitz}), there is a quasi-equivalence
$$
mf(Q, f)  \xra{\simeq}   D^b(R)/\Perf(R)
$$
of differential $\Z$-graded dg-categories, where $mf(Q, f)$ is regarded as $\Z$-graded by ``unfolding''.

A {\em free matrix factorization} is an object $(P,d)$ of $mf(Q,f)$ such that $P$ is a free $Q$ module of finite rank. 
Since $f$ is a non-zero-divisor, $\rk(P_0) = \rk(P_1)$. 
Upon choosing bases of these components, a free matrix factorization may thus be represented by 
a pair of $r \times r$ matrices, $(A,B)$,  with entries in $Q$ such that $AB = BA = fI_r$, where $r$ is the common rank of $P_0$ and $P_1$.

Since $mf(Q, f)$ is pretriangulated, $K_0(mf(Q,f)) = K_0^\Delta([mf(Q,f)]^{\idem})$. 
In particular, objects of $K_0(mf(Q,f))$ are represented by pairs $(P, e)$,
where $P \in mf(Q,f)$, and $e$ is an idempotent endomorphism of 
$P$ in the homotopy category $[mf(Q,f)]$. 
If $\Sing(R)$ consists of just one maximal ideal $\fm$, there is an equivalence
$$
[mf(Q,f)]^{\on{idem}} \cong [mf(\widehat{Q}, f)],
$$
where $\widehat{Q}$ denotes the $\fm$-adic completion of $Q$ (\cite[Theorem 5.7]{dyckerhoff}).


\section{Reduction to the case of a polynomial over $\C$}
\label{reduction}
In this section, we reduce the proof of Theorem \ref{introthm} to a special case.

\begin{prop} \label{prop1}
Theorem \ref{introthm} holds in general provided it holds in the following special case:
\begin{enumerate}
\item $k = \C$;
\item $Q = \C[x_0, \dots, x_n][1/h]$ for some odd integer $n$ and some polynomial $h$ 
such that $h(0, \dots, 0) \ne 0$, so that $U := \Spec(Q)$ is an affine  Zariski open neighborhood of the origin in $\A^{n+1}_\C$;
\item $f \in \C[x_0, \dots, x_n] \subseteq Q$; and
\item the only singular point of the morphism $f|_U: U \to \A^1_\C$ is the origin.
\end{enumerate}
\end{prop}

The proof will require a pair of lemmas.

\begin{lem} \label{lem2} 
Let $Q$ be a regular $k$-algebra and $f \in Q$ a non-zero-divisor such that $\Sing(Q/f)$ is a finite set of maximal ideals $\{\fm_1, \dots, \fm_m\}$. Suppose $h_1, \dots, h_m \in Q$ are such that 
$\fm_i \in \Spec(Q[1/h_i])$, and $\fm_j \notin \Spec(Q[1/h_i])$ for all $j \ne i$. 
Then the  natural dg functor
$$
mf(Q,f) \to \prod_{i=1}^m mf(Q[1/h_i], f)
$$
is a Morita equivalence.
\end{lem}

\begin{proof}
Let $Q'$ be the semi-localization of $Q$ at the list $\{\fm_1, \dots, \fm_m\}$, and let $Q_i'$ denote the localization of $Q$ at $\fm_i$. 
The natural maps
$$
mf(Q,f) \to mf(Q',f)  \and mf(Q[1/h_i], f) \to mf(Q'_i, f)
$$
are quasi-equivalences by \cite[Proposition 1.14]{orlov}, so it suffices to prove the functor
$$
mf(Q', f) \to \prod_{i=1}^m mf(Q'_i, f)
$$
is a Morita equivalence. Let $\widehat{Q'}$ denote the $\fm_1 \cap \cdots \cap \fm_m$-adic completion of $Q'$, and let $\widehat{Q'_i}$ denote the $\fm_i$-adic completion of $Q'_i$. The natural maps $Q' \to Q_i'$ induce maps $\widehat{Q'} \to \widehat{Q'_i}$, and the induced map
$$
\widehat{Q'} \to \widehat{Q'_1} \times \cdots \times \widehat{Q'_m}
$$
is an isomorphism. The bottom horizontal map in the commutative diagram
\begin{equation}
\label{productdiagram}
\xymatrix{
mf(Q', f)  \ar[r] \ar[d] & mf(\widehat{Q'}, f) \ar[d]  \\
\prod_{i = 1}^m mf(Q'_i, f)  \ar[r] & \prod_{i =1}^m mf(\widehat{Q'_i}, f)
}
\end{equation}
is a Morita equivalence by \cite[Theorem 5.7]{dyckerhoff}, and the right-most vertical map is an \emph{isomorphism} of dg-categories, i.e.~a dg functor such that the map on objects is a bijection and the maps on morphism complexes are isomorphisms. Thus, it suffices to show the top horizontal map is a Morita equivalence. 

Let $k_1^{\on{stab}}, \dots, k_m^{\on{stab}} \in mf(Q', f)$ denote the objects corresponding to the residue fields
$Q'/ \fm_1$, $\dots$, $Q'/ \fm_m$; we will use the same notation for the corresponding objects of $mf(\widehat{Q'}, f)$. (Here, ``stab" stands for
``stabilization".) 
By \cite[Theorem 3.5]{LP},
$\oplus_{i = 1}^m  k_i^{\on{stab}}$ is a generator of $[mf(Q', f)]$, and the corresponding object of $[mf(\widehat{Q'}, f)]$ is also a generator.
Finally, we claim that the natural map
$$
\End_{mf(Q', f)}(\oplus_{i = 1}^m  k_i^{\on{stab}}) \to  \End_{mf(\widehat{Q'}, f)}(\oplus_{i = 1}^m  k_i^{\on{stab}}) 
$$
is a quasi-isomorphism. The cohomology of the source (resp., target) computes the ``stable Ext" modules over $Q'/(f)$ (resp., $\widehat{Q'}/f$)
of the direct sum of the residue fields $k_1, \dots, k_n$ against itself.
The cohomology of the target is the $\fm_1 \cap \cdots \cap \fm_m$-adic completion of the cohomology of the source;
since the cohomology of the source is supported in $\fm_1 \cap \cdots \cap \fm_m$, the map to the completion is an isomorphism.
\end{proof}

\begin{lem} \label{lem822} Suppose $k$, $Q$ and $f$ are as in Theorem \ref{introthm}, and $\a$ belongs to $K_0(mf(Q,f))$. 
For any field extension $k \subseteq k'$, set $Q' = Q \otimes_k k'$ and $f' = f \otimes 1 \in Q$, and let $\a'$ be the image of $\a$ under
the natural map $K_0(mf(Q,f)) \to K_0(mf(Q',f'))$ induced by extension of scalars. 
Then,
\begin{enumerate}
\item $k'$, $Q'$ and $f'$ also satisfy the hypotheses of Theorem \ref{introthm},
\item $\chi(\a, \a)_{mf(Q,f)} = \chi(\a', \a')_{mf(Q',f')}$, and
\item $ch_{HP}(\a) = 0$ if and only if $ch_{HP}(\a') = 0$. 
\end{enumerate}

\end{lem}

\begin{proof}
The first assertion follows from the isomorphism $\Sing(Q'/f') \cong \Sing(Q/f) \times_{\Spec k} \Spec k'$ of $k'$-schemes
(using the reduced subscheme structures on $\Sing(Q'/f')$ and $\Sing(Q/f)$). 

For any finite length $Q$-module $N$, we have $N \otimes_Q Q' \cong N \otimes_k k'$ and hence $\dm_{k'}(N \otimes_Q Q') = \dm_k(N)$.
Given objects  $P_1, P_2$ of $[mf(Q, f)]$ equipped with idempotents $e_1, e_2$, set $P'_i = P_i \otimes_Q Q'$, $e'_i = e_i \otimes \id$.
Then we have a canonical isomophism
$$
\Hom_{[mf(Q', f')]^{\on{idem}}}((P'_1, e'_1), (P'_2 , e'_2 ))
\cong
\Hom_{[mf(Q, f)]^{\on{idem}}}((P_1, e_1 ), (P_2, e_2 )) \otimes_Q Q'. 
$$
This proves the second assertion. 

The final assertion holds since the map $ch_{HP}$ is natural, and the map on its targets
$$
HP_0(Q/k) \to HP_0(Q'/k')
$$
is injective, since  $HP_0(Q \otimes_k k'/k') \cong HP_0(Q/k) \otimes_k k'$.
\end{proof}

\begin{proof}[Proof of Proposition \ref{prop1}]
We split the proof into four steps. 
\vskip\baselineskip
\noindent {\bf Step 1: Reduction to the case where $f: \Spec(Q) \to \A^1_k$ has only one singular point.} Suppose $\Sing(Q/f) = \{\fm_1, \dots, \fm_m\}$ for $m
> 1$. By generic smoothness on the target, $f : \Spec(Q) \to \A^1_k$ has only finitely many critical values. 
(Note that this requires $\on{char}(k) = 0$.) 
Let $V \subseteq \A^1_k$ denote the Zariski open subset given by the complement of the nonzero critical values of $f$, and let $U \subseteq \Spec(Q)$ 
denote the fiber product $V \times_{\A^1_k} \Spec(Q)$; $U$ is an affine open subset of $\Spec(Q)$. By \cite[Proposition 1.14]{orlov}, there is a quasi-equivalence 
$$
mf(Q, f) \xra{\simeq} mf(U, f|_U)
$$
given by extension of scalars, and so, without loss of generality, we may assume the only critical value of $f: \Spec(Q) \to \A^1_k$ is 0, 
i.e.~the singular points of the morphism $f: \Spec(Q) \to \A^1_k$ coincide with the singular locus of $Q/f$.

Choose $h_1, \dots, h_m \in Q$ such that the only singular point of $f: \Spec(Q) \to \A^1_k$ on $\Spec(Q[1/h_i])$ is $\fm_i$, and set $Q_i := Q[1/h_i]$. By Lemma \ref{lem2},
we have a Morita equivalence of dg-categories
\begin{equation}
\label{morita}
mf(Q,f) \xra{\simeq} \prod_{i = 1}^m mf(Q_i, f).
\end{equation}
This induces an isomorphism of inner product spaces
$$
K_0(mf(Q,f)) \xra{\cong}
\bigoplus_i 
K_0(mf(Q_i,f)),
$$
where the source is equipped with the pairing $\chi(-,-)_{mf(Q,f)}$, and the target with the pairing $\oplus_i \chi(-,-)_{mf(Q_i,f)}$.
The Morita equivalence (\ref{morita}) also induces an isomorphism 
$$
HP_0(mf(Q,f))  \xra{\cong} \bigoplus_i HP_0(mf(Q_i, f)),
$$
and, by the naturality of the Chern character map, 
$ch_{HP(mf(Q,f))}$ corresponds to $\oplus_i ch_{HP(mf(Q_i,f))}$ under these isomorphisms.

\vskip\baselineskip
\noindent{\bf Step 2: Reduction to $k = \C$.} By Step 1,  we may assume $f$ has only one singularity.
We may find a subfield $k_0$ of $k$ having finite transcendence degree over $\Q$, a smooth $k_0$-algebra $Q_0$, an element $f_0 \in Q_0$ 
and a class $\a_0 \in K_0(mf(Q_0, f_0))$ 
so that
$Q = Q_0 \otimes_{k_0} k$, $f = f_0 \otimes 1$ and $\a_0 \mapsto \a$. Since $f$ has only one singular point, so does $f_0$. By Lemma \ref{lem822},
we may therefore assume $k$ has finite transcendence degree over $\Q$. 
Then there is an embedding $k \subseteq \C$, so we may apply Lemma \ref{lem822}  again to reduce to the case where $k =
\C$. Note that $f \otimes 1 \in Q \otimes_k \C$ may no longer have just one singularity, but we can apply Step 1 again, noting that the argument for that step
does not involve changing  the ground field.

\vskip\baselineskip  
\noindent{\bf Step 3: Reduction to the case where $Q = \C[x_0, \dots,
  x_n][1/h]$ for some $h$ and $f \in \C[x_0, \dots, x_n]$.}
Let $\widehat{Q}$ denote the completion of $Q$ at the singular point of $f$. By the Cohen Structure Theorem, there is a $\C$-algebra isomorphism 
$$
\Psi: \widehat{Q} \cong \C[[x_0, \dots, x_n]],
$$
and by ``finite determinacy" (see, for instance, \cite[Theorem 4.1]{greuel}), there is a $\C$-algebra automorphism $\Phi$ of $\C[[x_0, \dots, x_n]]$ 
such that $p := \Phi(\Psi(f)) \in \C[x_0, \dots, x_n]$.  Observe that 
$\C[[x_0, \dots, x_n]]/p$ has an isolated singularity. Applying the argument in Step 1, choose $h \in \C[x_0, \dots, x_n]$ such that $h$ does not vanish at the origin, and the only singularity of the map 
$$
p: \Spec(\C[x_0, \dots,  x_n][1/h]) \to \A^1_\C
$$ 
is at the origin. We have a chain of Morita equivalences
$$
mf(Q, f) \simeq mf(\widehat{Q}, f) \simeq mf(\C[[x_0, \dots, x_n]], \Psi(f)) \simeq mf(\C[x_0, \dots, x_n][1/h], p).
$$
The first and third Morita equivalences follow from \cite[Theorem 5.7]{dyckerhoff}, and the second is in fact an isomorphism of dg categories. The Morita invariance of $K_0$ and $HP_0$, along with the naturality of the Chern character, give the commutative diagram
\begin{equation}
\label{commdiag}
\xymatrix{
K_0(mf(Q,f)) \ar[r]^-{\cong} \ar[d]^-{ch_{HP}} & K_0(mf(\C[x_0, \dots, x_n][1/h], p)) \ar[d]^-{ch_{HP}}  \\
HP_0(mf(Q,f)) \ar[r]^-{\cong} & HP_0(mf(\C[x_0, \dots, x_n][1/h], p)),
} 
\end{equation} 
from which the claim follows.

\vskip\baselineskip
\noindent{\bf Step 4: Reduction to the case where $n$ is odd.}
When $n$ is even, $HP_0(mf(Q,f)) = 0$ by \cite[Theorem 6.6, Section 7]{dyckerhoff}.
\end{proof}

\section{The Milnor fibration}
\label{milnorbackground}

We now recall some background concerning the Milnor fibration. Everything in this section is likely well-known to experts in the field. We found Hertling's paper
\cite{Hertling}, his book \cite{hertlingbook},
and Kulikov's book \cite{Kulikov} to be particularly valuable. 

Throughout this section,
\begin{itemize}
\item $Q = \C[x_0, \dots, x_n][1/h]$ for some $h \notin \fm := (x_0, \dots, x_n)$,
\item $U = \Spec(Q)$, and 
\item $f \in Q$ is such that the only singularity of the morphism $f: U  \to \A^1_\C$ is at $\fm$.
\end{itemize}

\subsection{The Milnor fiber and its monodromy operator}
\label{milnorfiber}
Let $\epsilon, \eta$ be positive real numbers. Assume $\epsilon$ is chosen to be so small that $B_\epsilon \subseteq U$, where $B_\epsilon$ denotes the open ball in $\A^{n+1}_\C$ of radius $\epsilon$ centered at the origin. We set some more notation: 

\begin{itemize}
\item $T$ is the open disc of radius $\eta$ centered at the origin in $\A^1_\C$;
\item $X := f^{-1}(T) \cap B_\epsilon$;
\item by a slight abuse of notation, $f: X \to T$ denotes the map induced by $f$;
\item $T' := T \setminus \{0\}$;
\item  $f' : X' \to T'$ is the pullback of $f$, so that $X' = X
\setminus f^{-1}(0)$.
\end{itemize}

For $\epsilon$ small enough and $\eta \ll \epsilon$, $f'$ is a fibration, the \emph{Milnor fibration}. We will be interested in the fiber of $f'$,
the \emph{Milnor fiber}. Let $T_\infty \to T'$ be the universal cover of $T'$; explicitly, $T_\infty$ is a suitable open half plane. Let 
$f_\infty: X_\infty \to T_\infty$ be the pullback of $f'$. For each $t \in T'$, a choice of a lifting of $t$ to $\tilde{t} \in T_\infty$
determines a diffeomorphism $f^{-1}(t) \xra{\cong} f^{-1}_\infty(\tilde{t})$, and the inclusion map
$f_\infty^{-1}(\tilde{t}) \xra{\simeq} X_\infty$ is a homotopy equivalence. 
To avoid making a choice of fiber of $f'$,
we will consider the space $X_\infty$. By a famous theorem of Milnor \cite{milnor}, $X_\infty$ is homotopy equivalent to a wedge sum of $\mu$ copies of $S^n$, where
$$
\mu := \dim_\C Q/(  \frac{ \partial f }{ \partial x_0 }, \dots, \frac{ \partial f }{ \partial x_n } ) < \infty.
$$
In particular, its (reduced) cohomology is concentrated in degree $n$. Since $f'$ is a fibration over $T'$, $H^n(X_\infty ;  \Z)$ is equipped with a monodromy operator $M$. The group $H^n(X_\infty ;  \Z)$ equipped with its monodromy is a rich topological invariant of the morphism $f : U \to \A^1_\C$; we refer the reader to \cite[Chapter 3]{dimca} for a detailed discussion. For instance, by a theorem of Steenbrink (\cite{steenbrink}), the subgroup
$$
H^n(X_\infty; \Z)_1:= \bigcup_{j \geq 1}  \ker((M-\id)^j)
$$
may be equipped with a polarized mixed Hodge structure (PMHS) of level $n+1$ (see Appendix \ref{appendix} for background on PMHS's). The goal of the rest of this section is to describe Steenbrink's PMHS. 

\begin{rem}
\label{E913c}
When necessary, we write $T_\eta$, $X_{\e, \eta}$, $T'_\eta$, $X'_{\e, \eta}$, $X_\infty^{\e, \eta}$ and $T_\infty^{\eta}$ 
to indicate the dependence on these parameters. If $\e, \eta$ are chosen such that $X'_{\epsilon, \eta} \to T'_\eta$ is a fibration, and we have $\e' \leq \e$ and $\eta' \leq \eta$, then the induced maps from $T_{\eta'}$, $X_{\e', \eta'}$, etc. to
$T_\eta$, $X_{\e, \eta}$, etc. are all diffeomorphisms. In particular, we have an isomorphism
$$
H^n(X^\infty_{\e, \eta}; \Z) \xra{\cong}H^n(X^\infty_{\e', \eta'}; \Z). 
$$
\end{rem}

\subsection{The Gau\ss-Manin connection}
\label{GM}

To describe Steenbrink's polarized mixed Hodge structure on $H^n(X_\infty; \Z)_1$,
we realize $H^n(X_\infty; \C)_1$ as a subspace of a certain $D$-module $\cG_0$, the \emph{Gau\ss-Manin connection}, which we now describe. 
Our reference for this section is \cite[Section 4]{Hertling}.

The $n^{\th}$ higher direct image $\R^n f'_* \C_{X'}$ of $f': X' \to T'$ applied to the constant sheaf $\C_{X'}$  is a
complex vector bundle on $T'$ whose fiber over $t \in T'$ is $H^n(X_t; \C)$, where $X_t$ denotes the fiber of $X' \to T'$ over $t$. Let $\cE$ be the
sheaf of holomorphic sections of this bundle; that is, $\cE$ is the sheaf $\R^n f'_* \C_{X'}  \otimes_{\C_{T'}} \cO^{an}_{T'}$ of
$\cO^{an}_{T'}$-modules. 
Let $i: T' \into T$ be the inclusion, and define the sheaf $\cG := i_* \cE$
on $T$. For an open subset $V$ of $T$, $\Gamma(V, \cG)$ is a subspace of the collection of all functions 
sending $t \in V \setminus \{0\}$ to a class in $H^n(X_t; \C)$. 
Finally, define $\cG_0$ to be the stalk of $\cG$ at the origin. So, we may identify
$\cG_0$ as a subspace of the collection of germs of functions sending $t$ to a class in $H^n(X_t; \C)$, for $0 < || t || \ll 1$.

Write $\C\{t\}$ for the DVR consisting of power series in $t$ having a positive radius of convergence. By construction, $\cG_0$ is a $\C\{t\}[t^{-1}]$-vector space
(in fact, $\dim_{\C\{t\}[t^{-1}]}\cG_0 = \mu$, where $\mu$ is as defined in \ref{milnorfiber}). Moreover,  $\cG_0$ is a \emph{$D$-module};
that is, it is equipped with a $\C$-linear ``covariant differentiation'' endomorphism $\del_t$ satisfying
$\del_t t = \id + t \del_t$. In other words, $\cG_0$ is a module over the Weyl algebra $\C\{t\}\langle \partial_t \rangle$.
We will describe the operator $\del_t$ explicitly in Remark \ref{rem914} below.

For each complex number $\a$,  we define a $\C$-linear subspace
\begin{equation} \label{E129}
C_\a = \bigcup_{j \geq 1} \ker\left((t \del_t - \a)^j: \cG_0 \to \cG_0\right)
\end{equation}
of $\cG_0$. The following lemma explains the relationship between $\cG_0$ and $H^n(X_\infty; \C)_1$.
In this lemma and hereafter, we write $H^n_{deR}(Y; \C)$ for the $n$-th de Rham cohomology space of a complex manifold $Y$. 
In general, the complex vector space 
$H^n_{deR}(Y, \C)$ is defined to be the $n$-th hypercohomology of the holomorphic de Rham complex $\Omega^{*, \an}_{Y}$, but when
$Y$ is a Stein manifold, as it will be in all the cases that arise in this paper, it is given by the $n$-th cohomology of the global sections of the complex $\Omega^{*, \an}_{Y}$.
The isomorphism $H^n_{deR}(Y; \C) \cong H^n(Y; \C)$ is induced by the canonical map of complexes of sheaves $\C_Y \to \Omega^{*, \an}_{Y}$.
When $Y$ is Stein, it may also be defined by integrating closed forms along classes in $H_{n}(Y; \C)$.

\begin{lem}[Section 4 of \cite{Hertling}]  
\label{lem913}
There is an isomorphism of complex vector spaces
$$
\psi_0: H^n(X_\infty; \C)_1 \xra{\cong} C_0 \subseteq  \cG_0
$$
such that the composition of 
$$
H^n_{deR}(X'; \C) \cong H^n(X'; \C) \xra{\can} H^n(X_\infty; \C)_1 \xra{\psi_0} \cG_0
$$
sends $[\omega]$, for any $\omega \in \ker(\Omega^{\an, n}_{X'} \xra{d}\Omega^{\an, n+1}_{X'})$,
to the element of $\cG_0$ represented by the function
$$
t \mapsto [\omega|_{X_t}] \in H^n_{deR}(X_t; \C) \cong H^n(X_t; \C), \text{ for $0 < t \ll 1$}.
$$
\end{lem}

\subsection{The Brieskorn lattice}
\label{brieskorn}

In order to describe the PMHS on $H^n(X_\infty; \C)_1$, we will need to exploit some additional structure on the Gau\ss-Manin connection $\cG_0$: namely, an embedding of the \emph{Brieskorn lattice} 
$$
H_0'' := \frac{\Omega^{\an, n+1}_{X,0}}{df \smsh d \Omega^{\an, n-1}_{X,0}}
$$
in $\cG_0$, where $\Omega^{\an, j}_{X,0}$ is the stalk of $\Omega^{\an, j}_X$ at the origin. Our reference here is once again \cite[Section 4]{Hertling}. 

There is an injective map
\begin{equation} \label{E913a}
s_0: H_0'' \to \cG_0
\end{equation} 
defined by the formula
$$
s_0([\o]) = \left( t \mapsto \left[ \frac{\o}{df}|_{X_t}\right] \in H^n(X_t; \C) :  0 < |t| \ll 1\right) \in \cG_0,
$$
where the $\frac{\o}{df}$ is the Gelfand-Leray form of $\o$ (\cite[Section 7.1]{AGZV}). Here is a more precise definition of $s_0$: for any $\omega \in \Omega^{\an, n+1}_{X,0}$, there exists $N \gg 0$ such that $f^N \omega = df \smsh \beta$ for some $\beta \in \Omega^{\an, n}_{X,0}$. Choose an open neighborhood $V$ of the origin such that $\beta$ extends to an element of $\Gamma(V, \Omega^{\an, n})$. The element
$s_0([\omega])$ of $\cG_0$ is represented by the function
$$
t \mapsto [t^{-N} \beta|_{X_t}] \in H^n_{deR}(X_t; \C) \cong H^n(X_t; \C), \text{ for $0 < |t| \ll 1$}.  
$$
In order for this to make sense,  we need $X = X_{\e, \eta} \subseteq V$, but by using Remark \ref{E913c} we may assume $\e$ is small enough so that this holds. 

\begin{rem}
\label{lattice}
The image of $s_0$ is a ``lattice'' of $\cG_0$ in the sense that it is a free $\C\{t\}$-module of finite rank $\mu$, and 
$\im(s_0)[1/t] = \cG_0$.
\end{rem}

Equipping $H_0''$ with a $\C\{t\}$-action by letting $t$ act via multiplication by $f$ on $\Omega^{\an, n+1}_{X,0}$ makes $s_0$ a $\C\{t\}$-linear map. Define also a $\C$-linear endomorphism $\del^{-1}_t$ of $H_0''$ by
\begin{equation} \label{E913n}
\del_t^{-1}([\omega]) = [df \smsh \nu], \text{where $d \nu = \omega$}
\end{equation}
(using that the map $d: \Omega^{\an, n}_{X,0} \to \Omega^{\an, n+1}_{X,0}$ is surjective).
This makes $H_0''$ a module over the ring
$\C\{t\}\langle \del_t^{-1} \rangle$ defined by the relation $t \del_t^{-1} = \del_t^{-2} + \del_t^{-1} t$. The reason for the notation ``$\del_t^{-1}$" is that the  operator $\del_t$ acts invertibly on the image of $s_0$, and $s_0$ is a $\C\{t\}\langle \del_t^{-1} \rangle$-linear map. To explain this, we must introduce some more notation. 

Recall the subspaces $C_\a \subseteq \cG_0$ from (\ref{E129}).
Since $\del_t t = \id + t \del_t$, multiplication by $t$ induces a map
$t: C_\a \to C_{\a + 1}$ for each $\a$, and this map is an isomorphism. Similarly, multiplication by $\del_t$ induces a map $\del_t: C_{\a+1} \to C_\a$, and it is an isomorphism for all $\a \ne -1$. For all $\beta \in \mathbb{R}$, we define $\C\{t\}$-submodules
$$
V^{> \beta}  = \sum_{\b < \a} \C\{t\} C_\a  \and V^{\beta}  = \sum_{ \b \leq \a} \C\{t\} C_\a  
$$ 
of $\cG_0$. We are particularly interested in $\V$.
Upon restricting the indexing to $-1 < \a \leq 0$ in the definition of $\V$, we have an internal direct sum decomposition
$$
\V = \bigoplus_{-1 < \a \leq 0} \C\{t\} C_\a.
$$
Notice that $\del_t$ induces an isomorphism
$$
\del_t: V^{>0} \xra{\cong} \V;
$$
we equip $\V$ with the structure of a $\C\{t\}\langle \del_t^{-1} \rangle$-module by defining 
$$
\del_t^{-1}: \V \to \V
$$
to be the composition of the inverse of the isomorphism $\del_t: V^{>0} \xra{\cong} \V$ with the inclusion $V^{>0} \subseteq \V$. 

\begin{lem}[\cite{Hertling} Section 4] \label{lem914} 
The image of $s_0: H_0'' \to \cG_0$ is contained in $\V$, and the map $s_0: H_0'' \to \V$ is $\C\{t\}\langle \del_t^{-1} \rangle$-linear.
\end{lem}

\begin{rem} \label{rem914}
In fact, the lemma determines the operator $\del_t$ on $\cG_0$ completely. Since
$s_0$ induces an isomorphism
$$
H_0''[1/t] \xra{\cong} \cG_0
$$
(Remark \ref{lattice}), for any element $\b \in \cG_0$ we have $t^N \b = s_0(\a)$ for $N \gg 0$ and some $\a \in H_0''$. As discussed above in the definition of the map $s_0$, for any $\a \in H_0''$, it is known that
$t^M \a = \del_t^{-1} \g$ for $M \gg 0$ and some
$\g \in H_0''$. It follows that 
$$
t^L \b = s_0(\del_t^{-1} \gamma) \text{ for $L \gg 0$ and some $\g \in H_0''$.} 
$$
By the lemma, we get $s_0(\gamma) = \del_t s_0(\del_t^{-1} \gamma) = \del_t(t^L \b) = L t^{L-1} \b + t^L \del_t(\b)$,
and hence
$$
\del_t(\b) = t^{-L} s_0(\g) - L t^{-1} \b.
$$
\end{rem}

\subsection{Steenbrink's polarized mixed Hodge structure}
\label{PMHS}

We now describe Steenbrink's PMHS of level $n+1$ on $H^n(X_\infty; \Z)_1$, following Sections 3 and 4 of \cite{Hertling}. As discussed in Appendix \ref{appendix}, we must specify 
\begin{itemize}
\item an endomorphism $N$ of $H^n(X_\infty; \Q)_1$ such that $N^{n+2} = 0$,
\item a decreasing filtration $F^\bu$ on $H_\C$, and
\item a symmetric $\Q$-bilinear form
$$
S : H_\Q \otimes_\Q H_\Q \to \Q.
$$
\end{itemize}

The map $N$ is $-\on{Nilp(M)}$, where $\on{Nilp}(M)$ denotes the nilpotent part of the rational monodromy operator $M \otimes \Q$ on $H^n(X_\infty; \Q)$; the coefficient of $-1$ on $\on{Nilp}(M)$ doesn't appear in \cite{Hertling}, but this is due to an error which the author notes in \cite[Remarks 10.25]{hertlingbook}. The weight filtration $W_\bu$ on $H^n(X_\infty;\Q)$ is induced from $N$, as described
in Appendix \ref{appendix}. 
In particular, 
\begin{equation} \label{E914}
\im\left(H^n(X' ;\Q) \to H^n(X_\infty ; \Q)\right) = \ker(N) \subseteq W_{n+1} H^n(X_\infty ; \Q)_1.
\end{equation} 

The pairing
$$
S: H^n(X_\infty ; \Q)_1 \otimes_\Q H^n(X_\infty ; \Q)_1 \to \Q
$$
is described in (\cite[page 13]{Hertling}). It is non-degenerate (\cite[page 13]{Hertling}) and therefore induces an isomorphism
$$
\sigma : H^n(X_\infty ; \Q)_1 \xra{\cong} H_n(X_\infty ; \Q)_1.
$$
A formula for $S$, in terms of $\sigma$, is given as follows (\cite[page 14]{Hertling}):\begin{equation} \label{E914d}
S(a,b) = (-1)^{n(n-1)/2}  \sum_{m \geq 1} \ell(\frac{1}{m!} \sigma (N^{m-1}a), \sigma(b)),
\end{equation}
where $\ell$ is the Seifert pairing (\cite[page 40]{AGZV}).

\begin{rem}
 \label{E914g}
If $a \in H^n(X_\infty; \Q)$ is fixed by $M$, we have
$$
S(a,b) = (-1)^{n(n-1)/2}  \ell(\sigma(a), \sigma(b)).
$$
\end{rem}

To describe the filtration $F^\bu$ on $H^n(X_\infty; \C)_1$, we apply the discussions in \ref{GM} and \ref{brieskorn}:
$$
F^q H^n(X_\infty; \C)_1 = \psi^{-1}_0 \left( \frac{V^0 \cap \del_t^{n-q} s_0(H_0'') + V^{>0}}{V^{>0}}  \right).
$$
This description of $F^\bu$ in terms of the Gau\ss-Manin connection is due to work of Pham \cite{pham}, M. Saito \cite{saito}, Scherk-Steenbrink \cite{SS}, and Varchenko \cite{varchenko}.
\begin{rem}
\label{E914b}
Since $\im(\psi_0) \subseteq V^0$, we conclude that,
if $z \in H^n(X_\infty,\C)_1$ satisfies $\psi_0(z) \in \del_t^{n-q} \cdot s_0(H_0'')$, then $z \in
F^q H^n(X_\infty,\C)_1$.
\end{rem}

\begin{rem}
The isomorphism in Remark \ref{E913c} is an isomorphism of PMHS's. 
\end{rem}

\section{The map $ch_{X_\infty}$ and its properties} \label{sec:ch}

Throughout this section, we adopt the following notations and assumptions:

\begin{assumptions} \label{ass919}
Assume that 
  \begin{enumerate}
\item $n$ is an odd positive integer;
\item $Q = \C[x_0, \dots, x_n][1/h]$ for some $h \notin (x_0, \dots, x_n)$;
\item $f$ is an element of $\fm := (x_0, \dots, x_n) \cdot Q$ 
such that  the only singularity of the associated morphism $f: \Spec(Q) \to \A^1_\C$  of smooth affine varieties
is at $\fm$; and
\item $X = X_{\e_0, \eta_0}$, $X' = X'_{\e_0, \eta_0}$, etc., are defined from $f$  as in \ref{milnorfiber}
  for  sufficiently small parameters $0 < \eta_0 \ll \e_0 \ll 1$.  
\end{enumerate}
\end{assumptions}

\def\tQ{\widetilde{Q}}
\def\tfm{\widetilde{\fm}}

The goal of this section is to define a Chern-character-type map
$$
ch_{X_\infty}: K_0(mf(Q,f)) \to H^n(X_\infty; \Q)_1
$$
that satisfies  certain key properties; see Corollary \ref{cor913}. To define this map, we will utilize the henselization of $Q$ at $\fm$.
We start by recalling the relevant definitions.

Define $\C[[x_0, \dots, x_n]]^{alg}$ to be the collection of {\em algebraic power series} in $x_0, \dots, x_n$:
$$
\C[[x_0, \dots, x_n]]^{alg} = \{P \in \C[[x_0, \dots, x_n]] \mid g(P) = 0, \text{ for some $0 \ne g(t) \in \C[x_0, \dots, x_n][t]$} \}.
$$
Then the following properties hold (see, for instance, \cite[Lemma 2.29]{rond}):
\begin{enumerate}
\item There are inclusions
  $$
  Q \subseteq \C[[x_0, \dots, x_n]]^{alg} \subseteq \cO_{X,0}^{an}.
  $$
\item $\C[[x_0, \dots, x_n]]^{alg}$ is a hensel, regular local ring and its algebraic completion is $\C[[x_0, \dots, x_n]]$.
\item $\C[[x_0, \dots, x_n]]^{alg}$ is a filtered union of sub-$\C$-algebras $\tQ$ that are \etale extensions of $Q$. 
\end{enumerate}
The latter two properties listed amount to the fact that
$\C[[x_0, \dots, x_n]]^{alg}$ is the henselization of $Q$ at $\fm$ (or, equivalently, the henselization of $\C[x_0, \dots, x_n]$ at $(x_0, \dots, x_n)$).
For brevity, we write $Q^h = \C[[x_0, \dots, x_n]]^{alg}$ from now on.

We may describe $Q^h$ in more geometric language as follows:  Set $U = \Spec(Q)$, and let $u \in U$ be the closed point determined by $\fm$. 
Then $\Spec(Q^h)$ is isomorphic to the filtered limit
$$
U^h = \lim_{(V, v) \to (U,u)} V
$$
indexed by all pointed \etale neighborhoods $p: (V,v) \to (U,u)$.

To relate these two constructions, suppose we are given a pointed \etale neighborhood $p: (V,v) \to (U,u)$, and let $p_\C : V(\C) \to U(\C)$ 
denote the induced map on complex points. Then we may find
a pair $\e, \eta$ such that
$X_{\e, \eta}$ is contained in $X \cap \im(p_\C)$ (since $p_\C$ is an open mapping). The inclusion of $X_{\e, \eta}$ into $U(\C)$ then factors as 
\begin{equation} \label{E510}
X_{\e, \eta} \xra{\iota} V(\C) \xra{p_\C} U(\C),
\end{equation}
for a unique open inclusion $\iota$ of complex manifolds. 
Taking colimits of rings of functions realizes the inclusion
$Q^h \subseteq \cO_{X,0}^{an}$ above.

Put differently, suppose $Q \subseteq \tQ$ is an \etale extension of $Q$ that is contained in $Q^h \subseteq \cO_{X,0}^{an}$.
Since $\tQ$ is a finitely generated $\C$-algebra, and each generator converges on an open neighborhood of the origin, 
there exist sufficiently small $\e, \eta$ such that each element of $\tQ$ converges absolutely on $X_{\e, \eta}$. 
This induces the open inclusion $\iota: X_{\e, \eta} \into V(\C)$ above.

We next recall the classical Chern character map for $K_1$. 
For an essentially smooth $\C$-algebra $R$ and each $j \ge 0$, there is a map
$$
ch_1^{2j}: K_1(R) \to H^{2j-1}_{deR}(V(\C); \C),
$$
where $V = \Spec(R)$, 
that sends the class of an invertible  matrix $Y$ to the class of the $(2j-1)$-form 
$$
\frac{-1}{(2\pi i)^j}\frac{(j-1)!}{(2j-1)!} \tr(Y^{-1} dY (d(Y^{-1}) dY)^{j-1}) \in \Gamma(V, \Omega^{\an, 2j-1}).
$$
(Recall that, since $V(\C)$ is a Stein manifold, we may identify its de Rham cohomology with the cohomology of the complex
$(\Gamma(V(\C), \Omega^{\an, \bullet}), d)$.) 

Using $d(Y^{-1}) = - Y^{-1} (dY) Y^{-1}$, we can also write this as
$$
ch_1^{2j}([Y]) =
\frac{-(-1)^{j-1}}{(2\pi i)^j} \frac{(j-1)!}{(2j-1)!} \tr((Y^{-1} dY)^{2j-1}).
$$
The factor of $\frac{1}{(2 \pi i)^j}$ in this formula (which is not included by some authors) ensures that the image 
of the composition 
$$
K_1(R) \xra{ch_1^{2j}} H^{2j-1}_{deR}(V(\C); \C) \cong H^{2j-1}(V(\C); \C)
$$
lies in $H^{2j-1}(V(\C); \Q)$; see \cite[Section 2]{Pekonen} for a proof. Abusing notation a bit, we write $ch^{2j}_1$ also for the induced map
$$
ch_1^{2j}: K_1(R) \to H^{2j-1}(V(\C); \Q).
$$

Given an \etale extension $Q \subseteq \tQ$ of $Q$ inside $Q^h$, we proceed to define  a map
$$
\phi_{\tQ}: K_1(\tQ[1/f]) \to  H^n(X_\infty ; \Q)_1
$$
as follows. Let $V = \Spec(\tQ)$, a smooth affine complex variety,  and let $V'= \Spec(\tQ[1/f])$, an open subvariety of $V$. 
As noted above, the inclusion $\tQ \subseteq \cO_{X,0}^{an}$ determines an open inclusion $\iota: X_{\e, \eta} \into V(\C)$ of
complex manifolds for $\e, \eta$ sufficiently small. We
define $\phi_{\tQ}$ to be the composition of 
$$
K_1(\tQ[1/f])  \xra{ch^{n+1}_1} H^n(V'(\C); \Q) \xra{\iota^*}
H^n(X'_{\e, \eta} ; \Q) \xra{} H^n(X^{\e,\eta}_\infty ; \Q)_1 \xra{\cong} H^n(X_\infty ; \Q)_1, 
$$
where the penultimate map is pull-back along the canonical covering space map $X^{\e, \eta}_\infty \onto X'_{\e, \eta}$ and the final map
is the inverse of the isomorphism  $H^n(X_\infty ; \Q)_1 \xra{\cong}  H^n(X^{\e,\eta}_\infty ; \Q)_1$
given by  pull-back along the inclusion
$X^{\e, \eta}_\infty \subseteq X_\infty$.  It is clear from the construction that the map $\phi_{\tQ}$ is independent of the choice of $\e, \eta$.

\begin{prop} \label{prop915}
With the notation and assumptions listed above, there is a unique homomorphism of abelian groups
$$
ch_{X_\infty}: K^\Delta_0([mf(Q^h, f)]) \to H^n(X_\infty; \Q)_1
$$
such that the following property holds:
Given an \etale extension $Q \subseteq \tQ$, with $\tQ \subseteq Q^h$, and given a 
free matrix factorization $(A, B) \in mf(\tQ, f)$, we have
$$
ch_{X_\infty}([(A,B)_h]) =  \phi_{\tQ}([A]),
$$
where $(A,B)_h$ denotes the image of $(A,B)$ under the canonical map $mf(\tQ, f) \to mf(Q^h, f)$, and $[A]$ is the class in 
$K_1(\tQ[1/f])$  given by regarding $A$ as an invertible matrix with entries in $\tQ[1/f]$. 
\end{prop}

\begin{proof} 
  Given two \etale extension $Q \subseteq \tQ_1$ and $Q \subseteq \tQ_2$ inside $Q^h$ such that  $\tQ_1 \subseteq \tQ_2$,
  the composition of
$$
K_1(\tQ_1[1/f]) \to K_1(\tQ_2[1/f]) \xra{\phi_{\tQ_2}}  H^n(X_\infty ; \Q)_1
$$
coincides with $\phi_{\tQ_1}$. We thus obtain an induced map
$$
\colim K_1(\tQ[1/f]) \to H^n(X_\infty ; \Q)_1,
$$
where the colimit is indexed by all such \etale extensions. Since $K$-theory commutes with filtered colimits of rings, and
$\colim \tQ = Q^h$, which gives that $\colim \tQ[1/f] = Q^h[1/f]$, we obtain a map
$$
\phi_h: K_1(Q^h[1/f]) \to H^n(X_\infty, \Q)_1.
$$
The map $\phi_h$ is uniquely determined by the following property: for each \etale extension $Q \subseteq \tQ \subseteq Q^h$, the composition
$$
K_1(\tQ[1/f]) \to K_1(Q^h[1/f]) \xra{\phi_h} H^n(X_\infty, \Q)_1
$$
coincides with the map $\phi_{\tQ}$.

By \cite[Theorem 3.2]{weibelK}, we have an exact sequence
$$
K_1(Q^h) \to K_1(Q^h[1/f]) \xra{\del} G_0(Q^h/f) \to K_0(Q^h) \to K_0(Q^h[1/f]) \to 0
$$
such that, for any matrix factorization $(A, B) \in mf(Q^h, f)$, $\del([A]) = [\coker(A)]$. Since $Q^h$ is local, the last map is an isomorphism,
and so we obtain the right exact sequence
$$
K_1(Q^h) \to K_1(Q^h[1/f]) \xra{\del} G_0(Q^h/f) \to 0.
$$
Combining the canonical map $D^b(Q^h/f) \to D^b(Q^h/f) / \Perf(Q^h/f)$ with the quasi-equivalence 
$$
mf(Q^h, f) \xra{\simeq} D^b(Q^h/f) / \Perf(Q^h/f) 
$$
discussed in Section \ref{mf}, we obtain a map 
\begin{equation}
\label{G0K0}
G_0(Q^h/f) \to K_0^\Delta[(mf(Q^h, f))]
\end{equation}
whose kernel is generated by  $[Q^h/f]$.
We have $[Q^h/f] = \del([f])$, where $[f] \in K_1(Q^h[1/f]))$ is the class of $f$ regarded as a $1 \times 1$ invertible matrix.  
From this, we obtain a surjection
$$
\pi: K_1(Q^h[1/f]) \onto K^\Delta_0([mf(Q^h,f)])
$$
such that 
\begin{itemize}
\item if $(A,B) \in mf(Q^h, f)$, then $\pi([A]) =
[(A,B)]$, and
\item 
the kernel of $\pi$ is generated by $[f]$ and 
the image of
$(Q^{h})^\times \cong K_1(Q^h) \to K_1(Q^h[1/f])$.
\end{itemize}

We claim that $\phi_{h}$ annihilates the kernel of $\pi$.
For the generator $[f]$, this is obvious when $n \geq 3$, since   
$$
ch_1^{n+1}([f]) = \frac{-1}{(2 \pi i)^{p}} \frac{(p-1)!}{(2p -1)!}f^{-1}df (df^{-1} df)^{\frac{n-1}{2}} = 0.
$$ 
When $n = 1$, observe that, for any pair $(\epsilon, \eta)$ and $t \in T'_{ \eta}$, the restriction of the class
$f^{-1}df \in H^1_{deR}(X'_{\epsilon, \eta}; \C)$ to $H^1(X_{\epsilon, \eta} \cap f^{-1}(t) ; \C)$ is $0$, and hence $\phi_{h}([f]) = 0$.   
To show $\phi_h$ annihilates the image of $K_1(Q^h) \to K_1(Q^h[1/f])$, it suffices to prove it annihilates the image of
the composition $K_1(\tQ) \to K_1(Q^h) \to  K_1(Q^h[1/f])$ 
for each \etale extension $Q \subseteq \tQ \subseteq Q^h$.
This holds since  the composition of
$$
K_1(\tQ) \to K_1(\tQ[1/f]) \xra{ch_1^{n+1}} H^n(X'; \Q)
$$
factors through $H^n(B_\epsilon; \Q) = 0$, for $\e$ sufficiently small,
by the naturality of the Chern character.

It follows that $\phi_h$ factors through $\pi$ and induces  the map we seek:
$$
ch_{X_\infty}: K^\Delta_0([mf(Q^h, f)]) \to H^n(X_\infty; \Q)_1.
$$
This  map has the desired property by construction, and its uniqueness holds since every object in
$mf(Q^h, f)$ is of the form $(A, B)_h$ for some $\tQ$, $A$, and $B$ as above.
\end{proof}

\begin{thm} \label{thm915} Using Assumptions \ref{ass919},  set $p:=\frac{n+1}{2}$. Then the map $ch_{X_\infty}$
  defined in Proposition \ref{prop915} enjoys the following properties:
\begin{enumerate}
\item 
  For
  $(A,B) \in mf(Q^h, f)$, we have 
$$
\psi_0(ch_{X_\infty}(A,B)) =   \frac{1}{(2 \pi i)^{p}}\del_t^{p-1}s_0 \left(\frac{2 \tr((dA dB)^p)}{(n+1)!} \right),
$$
where $\psi_0$,  $\del_t$, and $s_0$ are as defined in Sections \ref{GM} and \ref{brieskorn}, and $dA, dB$ are
viewed as matrices with entries in $\Omega^{an, 1}_{X,0}$.

\item Using the notation of Section \ref{PMHS}, we have
$$
\im(\ch_{X_\infty}) \subseteq
\ker(N) \cap F^p H^n(X_\infty; \C)_1.
$$
In particular, $\im(\ch_{X_\infty}) \subseteq W_{n+1} H^n(X_\infty; \Q)_1$. 
\end{enumerate}
\end{thm}

\begin{rem}
  For any smooth complex variety $Y$, the image of the  map
  $$
  ch_1^{n+1}: K_1(Y) \to H^n(Y; \Q) \subseteq H^n(Y; \C)
  $$
  is contained in $F^p H^n(Y; \C)$, where $F^\bu$ is the Hodge filtration on $H^n(Y; \C)$  defined by Deligne (\cite{deligneII}, \cite{deligneIII}).
  The last statement of part (2) of Theorem \ref{thm915} would thus follow from the assertion that the canonical map
  $$
  H^n(V'; \Q) \to H^n(X_\infty; \Q)_1
  $$
  is a morphism of MHS's, where $V' = \Spec(\tQ[1/f])$ for any \etale extension $Q \subseteq \tQ$ contained in $Q^h$. 
Although this seems likely to be true, we were unable to prove it or find a reference for it,
  and we have opted for a more direct proof of the last statement of part (2).
\end{rem}

The proof of Theorem \ref{thm915} will use the following: 

\begin{lem} \label{lem913e}
Suppose $(A,B)$ is a free matrix factorization of $f$ in $R$, for an essentially  smooth $k$-algebra $R$  and a non-zero-divisor $f \in R$. 
Then, for any positive integer $j$, we have
$$
f \cdot \tr((dAdB)^j) = j df \smsh \tr(A dB (dA dB)^{j-1}) \in \Omega^{2j}_{R/k}.
$$
\end{lem}

\begin{proof} 
Since $f$ is a non-zero divisor on $R$, and $R$ is essentially smooth over $k$, the map
$\Omega^{2j}_{R/k}   \to \Omega^{2j}_{R/k}[1/f] \cong \Omega^{2j}_{R[1/f]/k}$ is an injection, and so without loss we may assume $f$ is a unit in $R$.  Then $A$ and $B$ are invertible matrices such that $A = fB^{-1}$, and so
$$
dA = f d(B^{-1})  + df B^{-1}.
$$
Using that $df \smsh df = 0$ and $\tr(XY) = (-1)^{pq} \tr(YX)$ if $X,Y$ are matrices with entries in $\Omega^p$ and  $\Omega^q$, respectively, we get
$$
\begin{aligned}
\tr((dA dB)^j) & = tr((f dB^{-1} dB + df B^{-1} dB)^j) \\
& =   f^j \tr((dB^{-1}dB )^j) + j f^{j-1} df \tr(B^{-1} dB (dB^{-1} dB)^{j-1}). \\
\end{aligned}
$$
Also, $dB dB^{-1} = - (dB B^{-1})^2$ and so 
$$
\tr((dB dB^{-1})^j) = (-1)^j \tr((dB B^{-1})^{2j}) = 0;
$$
the last equation holds since the trace of an even power of a matrix with odd degree entries over any graded commutative ring is 0. 
Using $B^{-1} = f^{-1}A$, $dB^{-1} = f^{-1} dA - df B^{-1}$, and $df \smsh df = 0$ therefore gives
$$
\begin{aligned}
\tr((dA dB)^j)  & = j f^{j-1} df \tr(B^{-1} dB (dB^{-1} dB)^{j-1}) \\
& = j f^{-1} df \tr(A dB (dA dB)^{j-1}). \qedhere
\end{aligned} 
$$
\end{proof}

\begin{proof}[Proof of Theorem \ref{thm915}]
  The second assertion follows from the first: by construction, $ch_{X_\infty}(A,B)$ belongs to
  $\im(H^n(X'; \Q) \to H^n(X_\infty; \Q)_1) = \ker(N)$, and, granting that (1) holds, 
$ch_{X_\infty}(A,B) \in F^p H^n(X_\infty; \C)_1$ follows from Remark \ref{E914b}.
The final clause follows from (\ref{E914}).

To prove (1), 
choose a lift of $(A,B) \in mf(Q^h,f)$ to an object of $mf(\tQ, f)$ for some \etale extension $Q \subseteq \tQ$ contained in $Q^h$.
Set $V = \Spec(\tQ)$ and $V' := \Spec(Q'[1/f])$. Without loss of generality, we may assume $X$ is sufficiently small so that every element of $\tQ$ converges
absolutely on $X$, and thus we may interpret $A, B$ as matrices with entries in $\Gamma(X, \cO_X^{an})$.

The composition 
$$
K_1(V') \xra{ch_1^{n+1}} H^n_{deR}(V'; \C) \to H^n_{deR}(X'; \C)
$$
sends $[A]$ to the class of
$$
\frac{-1}{(2 \pi i)^p}\frac{(p-1)!}{n!} \tr(A^{-1} dA (dA^{-1} dA)^{p-1}) \in \Gamma(X', \Omega_X^{\an, n}).
$$
Since $AB = f$ we have 
$$
dA^{-1} = f^{-1} dB - f^{-2} B df,
$$
and hence the image of $ch_1^{n+1}([A])$ in $H^n_{deR}(X'; \C)$ is the class represented by
$$
\frac{-1}{(2 \pi i)^p}\frac{(p-1)!}{n!} f^{-p} \tr(B dA (dB dA)^{p-1}) + df \smsh \omega
$$
for some $\omega \in \G(X', \Omega^{\an, n-1}_X)$. The composition
$$
H^n_{deR}(X'; \C)  \cong H^n(X'; \C) \to H^n(X_\infty; \C)
$$
sends $df \smsh \omega$ to $0$, 
because, for any $t \in T'$,  $X_\infty$ and $f^{-1}(t) \cap X'$ are homotopy equivalent,
and $df$ restricts to $0$ in the de Rham cohomology of $f^{-1}(t) \cap X'$. Thus,
$\psi_0(ch_{X_\infty} (A, B)) \in \cG_0$ is given by the function
$$
t \mapsto \frac{-1}{(2 \pi i)^p}\frac{(p-1)!}{n!} t^{-p} \tr(B dA (dB dA)^{p-1})|_{X_t}.
$$

By Lemma \ref{lem913e}, we have
$$
[f \tr((dA dB)^p)] = [p df \smsh \tr(A dB (dA dB)^{p-1})] \text{ in $H''_0$,} 
$$
and so  the definition of $s_0$ gives that  $s_0(\tr((dA dB)^p)) \in \cG_0$ is the function
$$
t \mapsto   p t^{-1}  \tr(A dB (dA dB)^{p-1})|_{X_t}.
$$

Since the class in $[(B, A)] \in K_0^\Delta([mf(Q^h, f)])$ is equal to $-[(A, B)]$,  we arrive at
\begin{equation} \label{E817b}
t^{p-1} \psi_0(\ch_{X_\infty} (A,B))   =  \frac{1}{(2 \pi i)^p}\frac{(p-1)!}{p \cdot n!}  s_0(\tr(dA dB)^p).
\end{equation}

By \cite[p. 16]{Hertling}, we have 
$$
\del_t \circ \psi_0 = - t^{-1} \circ \psi_0 \circ N/(2 \pi i),
$$
and hence
$$
\del_t \psi_0(ch_{X_{\infty}}(A,B)) = 0.
$$
By applying $\del_t^{p-1}$  to \eqref{E817b} we thus get
 $$
(p-1)!  \psi_0(\ch_{X_\infty}(A,B) )   =  \frac{1}{(2 \pi i)^p}\frac{(p-1)!}{p \cdot n!} \del_t^{p-1}  s_0(\tr(dA dB)^p), 
$$
which implies (1). 
\end{proof}

\begin{lem} 
\label{idemcompletion}
The canonical map
$$
[mf(Q, f)] \to [mf(Q^h, f)]
$$
exhibits $[mf(Q^h, f)]$ as the idempotent completion of $[mf(Q, f)]$. In particular, the canonical maps
  $$
K_0(mf(Q,f)) \to K_0(mf(Q^h, f)) \leftarrow K_0^\Delta([mf(Q^h, f)])
  $$
  are isomorphisms.
\end{lem}

\begin{proof}
Let $\widehat{Q}$ denote the $\fm$-adic completion of $Q$ (which coincides 
with the $\fm$-adic completion of $Q^h$). 
 By \cite[Proposition 1.14]{orlov} and \cite[Theorem 5.7]{dyckerhoff}, each of the canonical maps
\begin{equation}
\label{completion}
[mf(Q, f)] \to [mf(\widehat{Q}, f)] \leftarrow
[mf(Q^h, f)] 
\end{equation}
exhibit its target as the idempotent completion of its source. 
On the other hand,  the proof of \cite[Lemma 5.6]{dyckerhoff}, along with an application of \cite[Theorem 1.8]{LW}, shows that $[mf(Q^h, f)]$ is idempotent complete.
The first assertion follows. The rest follows from the results described in Section \ref{background}.
\end{proof}

Composing the map $ch_{X_\infty}$ defined in Proposition \ref{prop915} with the isomorphism 
$$
K_0(mf(Q, f)) \cong K_0^{\Delta}([mf(Q^h,f)])
$$ 
of the Lemma gives a map
\begin{equation} \label{E58b}
K_0(mf(Q, f)) \to H^n(X_\infty; \Q),
\end{equation}
which, abusing notation a bit, we also write as $ch_{X_\infty}$.

\begin{cor} \label{cor913}
In the setting of Theorem \ref{thm915}, for any $\a \in K_0(mf(Q,f))$, we have 
$$
S(ch_{X_\infty}(\a), ch_{X_\infty}(\a)) \geq 0,
$$ 
and
$$
S(ch_{X_\infty}(\a), ch_{X_\infty}(\a)) = 0 \text{ if and only if }
ch_{X_\infty}(\a) = 0,
$$
where $S$ is the pairing on $H^n(X_\infty; \Q)_1$ described in \eqref{E914d}.
  \end{cor}

\begin{proof} This follows from Theorem \ref{thm915} (2) and Lemma \ref{lem94} of the Appendix.
\end{proof}

\section{Positive semi-definiteness of the Euler pairing}
\label{Euler}

Throughout this section, we continue to operate under Assumptions \ref{ass919}. By \cite[4.1 and 4.4]{Hertling}, there is a ``higher residue'' pairing
$$
P_S: \V \times \V \to \C[[\del_t^{-1}]]
$$
taking values in the ring of formal powers series on the symbol $\del_t^{-1}$ 
that satisfies the following properties:
\begin{enumerate}
\item For elements $a, b \in H^n(X_\infty; \C)_1$, we have 
$$
P_S(\psi_0(a), \psi_0(b)) = 
\frac{-1}{(2\pi i)^{n+1}} S(a,b) \del_t^{-2}.
$$

\item For $[\omega], [\nu] \in H_0''$, we have
$$
P_S(s_0([\omega]), s_0([\nu])) = \Res_f(\omega, \nu) \cdot \del_t^{-n-1} 
 + \text{terms involving $\del_t^j$ for $j < -n-1$,} 
$$
where $\Res_f$ denotes the classical residue pairing on 
$$
\frac{\Omega^{\an, n+1}_{X,0}}{df \smsh  \Omega^{\an, n}_{X,0}}.
$$

\item For $\alpha, \beta \in \V$, 
$$
P_S(\del_t^{-1} \a, \b) = \del_t^{-1} P_S(\a,\b)
$$
and
$$
P_S(\a, \del_t^{-1}  \b) = - \del_t^{-1} P_S(\a,\b).
$$
\end{enumerate}
\begin{rem}
Item (1) above differs from \cite[Definition 4.1]{Hertling} by a sign. This is due to an error in loc. cit., which is corrected in \cite[(10.83)]{hertlingbook}.\end{rem}

As above, set $p = \p$. For any object $(A, B) \in mf(Q^h, f)$, define 
\begin{equation} \label{917}
ch_{PV}(A,B): =\frac{2 \tr(dA dB)^p}{(n+1)!} \in \Omega^{\an, n+1}_{X,0};
\end{equation}
the reason for the choice of notation $ch_{PV}$ will be made clear below. By Theorem \ref{thm915}, we have 
$$
s_0(ch_{PV}(A,B)) = (2\pi i)^p \del_t^{-p+1}  \psi_0(ch_{X_\infty}(A,B)),
$$
and so the third property of $P_S$ listed above implies
$$
P_S (s_0(ch_{PV}(A,B)), s_0(ch_{PV}(A',B'))) =
(-1)^{p-1} (2 \pi i)^{n+1} P_S(\psi_0(\ch_{X_\infty}(A,B), \ch_{X_\infty}(A', B'))) \del_t^{-n+1}
$$
for any $(A,B), (A', B') \in mf(Q^h, f)$. From the first two properties of $P_S$, we get
$$
\begin{aligned}
P_S (s_0(ch_{PV}(A,B)), s_0(ch_{PV}(A',B'))) =  
\Res_f & \left(ch_{PV}(A,B), ch_{PV}(A', B')   \right) \del_t^{-n-1} \\
& + \text{ terms of lower degree,} \\
\end{aligned}
$$
and
$$
P_S(\psi_0(\ch_{X_\infty}(A,B)), \psi_0(\ch_{X_\infty}(A', B')))  \del_t^{-n+1} =
\frac{-1}{(2 \pi i)^{n+1}} S(\ch_{X_\infty}(A,B), \ch_{X_\infty}(A', B')) \del_t^{-n -1}.
$$
By comparing coefficients, we deduce
\begin{equation} \label{E916}
\Res_f(ch_{PV}(A,B), ch_{PV}(A',B'))  =  (-1)^p S(\ch_{X_\infty}(A,B), ch_{X_\infty}(A',B')). 
\end{equation}
  
Finally, by Polishchuk-Vaintrob's Hirzebruch-Riemann-Roch formula for matrix factorizations (\cite[(0.2), (0.5), (0.6)]{PV}), we have
\begin{equation}
\label{PVresult}
\chi((A,B), (A',B'))_{mf(Q, f)}   = (-1)^{{n + 1 \choose 2}} \Res_f( ch_{PV}(A,B), ch_{PV}(A',B')).
\end{equation}
(Recall that, by Lemma \ref{idemcompletion}, any class in $K_0(mf(Q, f))$ may be represented by a free matrix factorization in $mf(Q^h, f)$.)
Combining (\ref{E916}) and (\ref{PVresult}), and noting that $(-1)^{{n + 1 \choose 2}}  =(-1)^p$, we obtain the following analogue of Polishchuk-Vaintrob's Hirzebruch-Riemann-Roch formula:

\begin{thm}
\label{BvS}
For any matrix factorizations $(A, B), (A', B') \in mf(Q, f)$, we have
$$\chi((A,B), (A',B'))_{mf(Q,f)} =  S(\ch_{X_\infty}(A,B), ch_{X_\infty}(A',B')).$$
\end{thm}

\begin{cor} \label{cor819} With the notation as above, given $\a  \in K_0(mf(Q, f))$ we have
$$
\chi(\a, \a)_{mf(Q,f)}  \geq 0,
$$
and 
$$
\chi( \a, \a)_{mf(Q,f)}  = 0 \, \text{ if and only if } \,  ch_{X_\infty}(\a) = 0.
$$
\end{cor}

\begin{rem} It is asserted in \cite{MPSW} that the ``Herbrand difference'' $h$ is negative definite for the graded case of this Corollary.
  The Herbrand difference coincides with the Euler pairing; see the proof of Corollary \ref{chitheta} below.
  But \cite{MPSW} is incorrect, since the authors overlooked the minus sign in $\theta(M, M') = - h(M^*, M')$; see \cite[Corollary 6.4.1]{buchweitz}. The same
  error occurs in \cite[Remark 1.2]{BvS}. 
\end{rem}

\section{Proof of the main theorem}
Throughout this section, we continue to operate under Assumptions \ref{ass919}. As in the introduction,
we let  $HP( - )$ denote the periodic cyclic homology functor for dg-categories, and we write $HH( - )$ and $HN( - )$ for the Hochschild
and negative cyclic homology functors.  As discussed in the introduction, there is a Chern character map
$$
ch_{HP} : K_0( - ) \to HP_0( - ),
$$
and there are analogous maps $ch_{HN}$, $ch_{HH}$. There are canonical natural transformations $HN_0 \to HH_0$ and $HN_0 \to HP_0$, and each of $ch_{HP}$ and $ch_{HH}$
factor through $ch_{HN}$. 
Note that the subscript $0$ on $HP, HH,$ and $HN$ is to be understood modulo $2$,
since we are working with differential $\Z/2$-graded categories.

\begin{prop}
\label{HN}
Let $Q$ and $f$ be as in Assumptions \ref{ass919}. As in Section \ref{brieskorn}, let $H_0''$ denote the Brieskorn lattice, and let $\widehat{H_0''}$ denote its $\fm$-adic completion. There is a canonical isomorphism
$$
\widehat{H_0''} \xra{\cong} HN_0(mf(Q,f))
$$
such that, for any class $[(A, B)] \in K_0^{\Delta}([mf(Q^h,f)]) \cong K_0(mf(Q, f))$, the class in $H_0''$ represented by $ch_{PV}(A,B)$ (as defined in \eqref{917})
is sent, under the composition
\begin{equation}
\label{brieskornmap}
H_0'' \to \widehat{H_0''} \xra{\cong} HN_0(mf(Q,f)),
\end{equation}
to $ch_{HN}(A, B)$. 
\end{prop}

\begin{rem}
As an aside, we note that, by \cite[Theorem 1]{shklyresidue} and \cite[Theorem 1.8]{BW2}, the map (\ref{brieskornmap}) identifies, up to multiplication by the sign $(-1)^{\binom{n}{2}}$, the higher residue pairing on $H_0''$ with the natural pairing on $HN_0(mf(Q,f))$ defined as in \cite[line (6)]{shklyresidue}.
\end{rem}

\begin{proof}
By \cite[Proposition 3.14]{efimov} and \cite[Section 4.8]{PP}, there is an HKR-type isomorphism
$$
HN(mf(Q,f)) \cong (\Omega^\bullet_{Q/\C}[[u]], ud - df)
$$
in the derived category of $\Z/2$-graded $k$-vector spaces. Here, the target is the ($\Z/2$-graded folding of the) direct product totalization of the upper-half plane bicomplex $B^{*,*}$, where $B^{p,q} = \Omega^{p+q}_{Q/\C} u^q$, the horizontal differential is given by $- \wedge df$, and the vertical differential is given by the de Rham differential $d$.
Because of the isolated singularity assumption, each row is exact everywhere except in the right-most position. Considering the spectral sequence associated to the filtration of $B^{*,*}$ by rows (\cite[Definition 5.6.2]{weibel}), one sees 
$$
HN_{0} (mf(Q,f)) \cong \coker\left(\Omega^n_{Q/\C}[[u]] \xra{ud - df} \Omega^{n+1}_{Q/\C}[[u]]\right).
$$

Let $\hQ$ be the completion of $Q$ at $\fm$, set $\hOmega^p = \Omega^p_{Q/\C} \otimes_{Q} \hQ$, and form the analogous 
bicomplex $\hB^{*,*}$ with direct product totalization
$(\hOmega^\bullet[[u]], ud - df)$. The map $B^{*,*} \to \hB^{*,*}$ is a quasi-isomorphism along each row, and so, by comparing spectral sequences associated to these bicomplexes, one sees
the canonical map
$$
(\Omega^\bullet_{Q/\C}[[u]], ud - df) \to (\hOmega^\bullet[[u]], ud - df)
$$
is a quasi-isomorphism.
We thus have an isomorphism
\begin{equation} \label{E819b}
HN_0 (mf(Q,f)) \cong \coker\left(\hOmega^n[[u]] \xra{ud - df} \hOmega^{n+1}[[u]]\right).
\end{equation}
Moreover, by \cite[Example 6.1]{BW}, for any $(A,B) \in mf(Q^h, f)$,
$ch_{HN}(A,B) \in HN_0(mf(Q^h,f)) \cong HN_0(mf(Q, f))$ corresponds, under \eqref{E819b}, to the class represented by
$$
\frac{2}{(n+1)!} \tr((dA dB)^\p) \in \hOmega^{n+1}.
$$

By \cite[Section 2]{schulze} (see also the discussion in \cite[Section 5]{shklyarovhodge}), there is a canonical isomorphism
$$
\widehat{H_0''} \xra{\cong} \coker\left(\hOmega^n[[u]] \xra{ud - df} \hOmega^{n+1}[[u]]\right),
$$ 
and the composition
$$
H_0''  \to \widehat{H_0''}  \xra{\cong} \coker\left(\hOmega^n[[u]] \xra{ud - df} \hOmega^{n+1}[[u]]\right) \\
 \cong HN_0(mf(Q, f))
$$
has the desired property.
\end{proof}

\begin{thm} \label{thm822}
The following statements are equivalent:
  \begin{enumerate}
  \item   $\chi( \alpha, \alpha )_{mf(Q, f)} = 0$.
  \item $ch_{X_\infty}(\alpha) = 0$ in $H^n(X_\infty; \Q)_1$.
  \item $ch_{HN}(\alpha) = 0$ in $HN_0(mf(Q,f))$.
  \item $ch_{HP}(\alpha) = 0$ in $HP_0(mf(Q,f))$.
  \item $ch_{HH}(\alpha) = 0$ in $HH_0(mf(Q,f))$.
  \end{enumerate}
\end{thm}

\begin{proof} The equivalence of (1) and (2) is implied by Corollary \ref{cor819}. (2) $\Rightarrow$ (3) follows from Theorem \ref{thm915} and Proposition \ref{HN}.  (3) $\Rightarrow$ (4) and (3) $\Rightarrow$ (5) hold since $ch_{HP}$ and $ch_{HH}$ factor through $ch_{HN}$, and (4) $\Rightarrow$ (3) holds since the canonical map 
$$
HN_0(mf(Q,f)) \to
  HP_0(mf(Q,f))
$$ 
is an injection; as discussed in \cite[page 12]{shklyresidue}, this is equivalent to noncommutative Hodge-to-de Rham degeneration for $mf(Q, f)$,
which holds by \cite[Section 7]{dyckerhoff}. Finally, (5) $\Rightarrow$ (1) follows from \cite[Theorems 2 and 3]{shklyHRR}.
(Note that, by \cite[Theorem 5.2]{dyckerhoff}, $mf(Q, f)$ is Morita equivalent to a dga, so the results of \cite{shklyHRR} apply.)
\end{proof}

\begin{proof}[Proof of Theorem \ref{introthm}]
Apply Proposition \ref{prop1}, Corollary \ref{cor819}, and Theorem \ref{thm822}.
\end{proof}

\section{Hochster's theta pairing}
\label{hochstersection}
We now apply Corollary \ref{cor819} to prove Theorem \ref{hochster} from the introduction. Let $Q$ and $f$ be as in the statement of Theorem \ref{hochster}. Set $R = Q/f$, $Y = \Spec(Q)$, and $Z = \Spec(R)$.
\subsection{Background on intersection theory}
Let $K_0^{Z}(Y)$ denote the Grothendieck group of the triangulated category
of perfect complexes on $Y$ with support in $Z$. 
Let 
$$
ch_Y^Z : K_0^Z(Y) \to A_*(Z) \otimes \Q
$$
denote the \emph{localized Chern character}, as defined in \cite[Definition 18.1]{fulton}. Here, $A_i( - )$ denotes the group of dimension $i$ cycles modulo rational equivalence. Note that $ch_Y^Z$ is an isomorphism upon tensoring with $\Q$. Let $ch_i$ denote the composition of $ch_Y^Z$ with the projection of $A_*(Z)$ onto $A_i(Z)$. Gillet-Soul\'e define Adams operations $\psi^l$ on $K_0^Z(Y)$ for $l \ge 0$ in \cite[Section 4]{GS}; by a theorem of Kurano-Roberts (\cite[Theorem 3.1]{KR}), we have
$$
ch_i \circ \psi^l = l^{d - i} ch_i
$$
for $i \ge 0$ and $l \ge 1$, where $d = \dim(Q)$. 

The map
$$
\tau_Z : G_0(Z) \to A_*(Z) \otimes \Q
$$
defined in \cite[Section 18.2]{fulton} is also an isomorphism upon tensoring with $\Q$. Let $\tau_i$ denote the composition of $\tau_Z$ with the projection of $A_*(Z)$ onto $A_i(Z)$. There are also Adams operations $\psi_l$ defined on $G_0(Z)$ (\cite{soule}), and, by \cite[Proposition 2.4]{haution}, we have 
$$
\tau_i \circ \psi_l = l^{-i} \tau_i.
$$
There is an isomorphism
$$
r : K_0^Z(Y) \xra{\cong} G_0(Z)
$$
(\cite[Lemma 1.9]{GS}), and, by \cite[line (5)]{haution}, we have
\begin{equation}
\label{compatible}
\psi_l \circ r = l^{-d} (r \circ \psi^l).
\end{equation}
Let 
$$
G_{0, (i)}(Z) \subseteq G_0(Z) \otimes \Q
$$ 
be the inverse image of $\tau_i \otimes \Q$ and let
$$
K^Z_{0, (i)}(Y) \subseteq K_0^Z(Y) \otimes \Q
$$
be the inverse image of $ch_i \otimes \Q$. Using \eqref{compatible} we obtain
\begin{equation}
\label{summands}
r(K^Z_{0, (i)}(Y)) = G_{0, (i)}(Z).
\end{equation}

\subsection{Proof of Theorem \ref{hochster}}
Let
$$
\theta : G_0(Z) \times G_0(Z) \to \Z
$$ 
denote the Hochster theta pairing, as defined in Section \ref{application}. A conjecture of Dao-Kurano (\cite[Conjecture 3.4]{DK}) 
predicts that $\theta( \a, - ) : G_0(Z) \otimes \Q \to \Z$ is the zero map for any $\a \in G_{0, (i)}(Z) $ when $i \ne \frac{\dim(Q) }{2}$; in particular, if $\dim(Q)$ is
odd, $\theta$ vanishes. 
The first goal of this section is to prove this conjecture in the case where $R$ is a complex hypersurface:

\begin{thm}
\label{adams}
Suppose $g \in \C[x_0, \dots, x_n]$ is such that $g(0, \dots, 0) \ne 0$, and
let $f \in \fm = (x_0, \dots, x_n) \subseteq \C[x_0, \dots, x_n][1/g]$ be such that the singular locus of $R := \C[x_0, \dots, x_n][1/g]/(f)$ 
consists of only the maximal ideal $\fm$.  Then 
$$
\theta(\a, - ): G_0(Z) \otimes \Q \to \Z
$$
is the zero map whenever $\a \in G_{0, (i)}(Z) $ for $i \ne \frac{n+1}{2}$; in particular, it is $0$ for all $\a$ when $n$ is even.
\end{thm}

\begin{proof}
As discussed in Section \ref{application}, the case where $n$ is even was proven independently by Buchweitz-van Straten and Polishchuk-Vaintrob (\cite{BvS}, \cite{PV}), so we may assume $n$ is odd. By \cite[Section 4]{brown} and \cite[Propositions 4.1 and 4.2]{BvS}, there is a map $\gamma : K_0^{Z}(Y) \to KU(B_\epsilon, X_\infty)$, where $B_\epsilon$ and $X_\infty$ are chosen as in Section \ref{milnorfiber}, such that
\begin{enumerate}
\item $\gamma$ commutes with $p^{\on{th}}$ Adams operations for any prime $p$, and 
\item if $\gamma(\b) = 0$, $\theta(r(\b), -) : K^Z_0(Y) \to \Z$ is the zero function.
\end{enumerate}
Observe that, since the (reduced) cohomology of $X_\infty$ is concentrated in degree $n$, the only nonzero eigenspace of $KU(B_\epsilon, X_\infty)$ for any Adams operation $\psi^l$ is the one corresponding to the eigenvalue $l^{\frac{n+1}{2}}$. The statement now follows from \eqref{summands}. 
\end{proof}

As a corollary, we prove \cite[Conjecture 3.1(4)]{DK} in the setting of Theorem \ref{adams}:
\begin{cor}
\label{theta}
If $R$ is as in the statement of Theorem \ref{adams}, and $M, M'$ are maximal Cohen-Macaulay $R$-modules,
then 
$$
\theta(M^*, M') = -(-1)^{\frac{\dim(Q)}{2}} \theta(M, M').
$$
\end{cor}

\begin{proof}
Combine \cite[Example 18.3.19]{fulton} and Theorem \ref{adams}. 
\end{proof}

Let $\sigma$ denote the map $G_0(Z) \to K_0(mf(Q, f))$ defined in the same way as the map in (\ref{G0K0}).

\begin{cor}
\label{chitheta}
If $Q = \C[x_0, \dots, x_n][1/g]$ and $f$ are as in Theorem \ref{adams}, we have
$$
\chi(\sigma( - ), \sigma( - ))_{mf(Q, f)} = (-1)^{ \frac{\dim(Q)}{2}} \theta( - , -)  .
$$
\end{cor}

\begin{proof}
  Let $M$ and $M'$ be maximal Cohen-Macaulay $R$-modules. The \emph{Herbrand difference} of $M$ and $N$ is the integer
$$
\begin{aligned}
  h(M, M') & := \len_R \Ext^2_R(M, M') -  \len_R \Ext^1_R(M, M') \\
           &  = \dm_\C \Ext^2_R(M, M') -  \dm_\C \Ext^1_R(M, M').
\end{aligned}
$$

It follows from \cite[Corollary 6.4.1]{buchweitz} that
$$
\theta(M, M') = -h(M^*, M'),
$$ 
where $M^*$ denotes the $R$-linear dual of $M$. (Beware that in both \cite{MPSW} and \cite[Remark 1.2]{BvS} the sign is omitted in this formula.)
Moreover, 
$$
\chi(\sigma( M), \sigma(M'))_{mf(Q,f)} = h(M, M'). 
$$
The result therefore holds for $M, M'$ maximal Cohen-Macaulay, 
by Corollary \ref{theta}. Since $G_0(R)$ is generated as an abelian group by classes of maximal Cohen-Macaulay $R$-modules, the proof is complete.
\end{proof}

\begin{proof}[Proof of Theorem \ref{hochster}] 
Using arguments similar to those in Section
\ref{reduction},  we may reduce to the case when $Q$ and $f$ are as in Theorem \ref{adams}. 
Then the result follows from  Corollaries \ref{cor819} and \ref{chitheta}.
\end{proof}

\begin{rem}
  We observe that each of Conjectures (1) -- (5) in \cite[Conjecture 3.1]{DK}
  is now proven over the complex numbers: to summarize, (1) is proven independently by Buchweitz-van Straten \cite{BvS} and Polishchuk-Vaintrob \cite{PV}; (2) and (3) are
  established in \cite{BMTW} and \cite{brown}, respectively; and (4) and (5) follow from Corollary \ref{theta} and Theorem \ref{hochster}, respectively.
  (As noted in the introduction, (5) was originally posed in \cite{MPSW}.)
\end{rem}

\appendix
\section{Polarized mixed Hodge structures} \label{appendix}
Our reference here is \cite[Section 2]{Hertling}. Let $V$ be a finite dimensional vector space over a field $k$, and let $m$ be a non-negative integer. For any $k$-linear endomorphism $N$ of $V$ such that $N^{m+1} = 0$, there is a unique 
increasing filtration $W_\bu$ of $V$ of the form
$$
0 = W_{-1} \subseteq W_0 \subseteq \cdots \subseteq W_{2m} = V
$$
such that $N(W_l) \subseteq W_{l-2}$ and the induced map $N^l:  \Gr_{m+l}^W \to  \Gr_{m-l}^W$ is an isomorphism for all $l \geq 0$ (here, $\Gr_j^W := W_j/W_{j-1}$).  
$W_\bu$ is called the \emph{weight filtration} of $V$ associated to $(N, m)$. Note that the filtration depends not just on $N$, but also on
the specified integer $m$.

The weight filtration is natural, in the sense that if $g: V \to V'$ is a $k$-linear transformation, $N, N'$ are endomorphisms of $V, V'$ whose $(m+1)^{\on{st}}$ power is 0, and $N' \circ g = g \circ N$, then $g(W_j(V)) \subseteq W_j(V')$ for all $j$.

\begin{defn} For a non-negative integer $m$, a {\em mixed Hodge structure of level $m$} (MHS of level $m$, for short) consists of the following data:
\begin{itemize}
\item a finitely generated abelian group $H$ with associated vector spaces $H_\Q = H \otimes_\Z \Q$ and $H_\C = H \otimes_\Z \C$, 
\item a $\Q$-linear endomorphism $N: H_\Q \to H_\Q$ such that $N^{m+1} = 0$, and  
\item a decreasing filtration $F^\bu$ of the complex vector space  $H_\C$.
\end{itemize}
These data are required to satisfy two conditions.
Let $W_\bu$ be the weight filtration of $H_\Q$ associated to $(N, m)$ as defined above, and, for each $j$,  let $F^\bu \Gr_j^W$ be the decreasing
filtration of $\Gr_j^W \otimes_\Q \C$ induced by $F^\bu$. Then

\begin{itemize}
\item For all $j$, 
$\Gr_j^W \otimes_\Q \C$ may be written as an internal direct sum $F^p \Gr_j^W \oplus \overline{F^q \Gr_j^W}$ for all integers $p,q$ such that $p + q = j+1$, and 
\item $N(F^p) \subseteq F^{p-1}$ for all $p$.
\end{itemize}
Here, the $\overline{( - )}$ notation indicates complex conjugation: given an element of $F^q \Gr_j^W$ of the form $v \otimes z$, its conjugate is $v \otimes \overline{z}$. By abuse of notation, we usually write an MHS of level $m$, which consists of the triple of data $(H, N, F^\bu)$, by just $H$. 

A morphism of MHS's of level $m$, from $H$ to $H'$, is a homomorphism of the underlying  abelian groups $g: H \to H'$  such that $N_{H'} \circ g = g \circ N_H$ and 
$g(F^p H_\C) \subseteq F^p H'_\C$ for all $p$.  
\end{defn}

\begin{ex}
For a non-negative even integer $m$ and a finitely generated abelian group $A$, the {\em trivial MHS of level $m$} on $A$ is defined by taking 
$N = 0$ and 
$$
F^p = 
\begin{cases}
A_\C, & \text{if $p \leq m/2$, and} \\
0, & \text{if $p > m/2$.} \\
\end{cases}
$$
Note that 
$$
W_j = 
\begin{cases}
0, & \text{if $j \leq m-1$, and} \\
A_\Q, & \text{if $j \geq m$} \\
\end{cases}
$$
so that $\Gr_m^W A_\Q = A_\Q$ and $\Gr_j^W A_\Q = 0$ for all $j \ne m$. The induced Hodge filtration on $\Gr_m^W(A_\C) = A_\C$  is given by
$$
F^p \Gr_m^W(A_\C) = 
\begin{cases}
A_\C, & \text{if $p \leq m/2$, and} \\
0, & \text{if $p > m/2$.} \\
\end{cases}
$$
Since $p+q = m + 1$ implies that either $p > m/2$ and  $q \leq m/2$ or vice versa, the only non-trivial axiom is satisfied. 

If $H$ is an MHS of level $m$, then a morphism $g$ of MHS's of level $m$ from $A$ to $H$ is the 
same thing as a homomorphism of ordinary abelian groups from $A$ to the subgroup
$$
\ker(N: H_\Q \to H_\Q) \cap F^{m/2} H_\C  \cap H
$$
of $H$. It follows from the axioms that $\ker(N) \subseteq W_m H_\Q$, which means $g$ induces a map
$$
\Gr^W_m(g): A \to \Gr_m^W(H).
$$
More generally, if $A$ is any (not necessarily finitely generated) abelian group, and $g: A \to H_\Q$ is a group homomorphism whose image is contained in $\ker(N) \cap F^{m/2} H_\C$, then $g$ induces a map $A \to  \Gr_m^W(H)$, which we also denote by $\Gr^W_m(g)$. 
\end{ex}

\begin{lem} \label{lem815}
Suppose $m$ is an even integer, and $H = (H, N, F^\bullet)$ is an MHS of level $m$. If $A$ is a (not necessarily finitely generated)
abelian group, and $g: A \to H_\Q$ is a homomorphism of groups whose image is contained in $\ker(N) \cap F^{m/2} H_\C$, then
$\ker(g) = \ker(\Gr_m^W(g))$.
\end{lem}

\begin{proof}
If $\Gr_m^W(g)(a) = 0$,  then $g(a) \in W_{m-1} H_\Q \cap F^{m/2} H_\C$. But the axioms of an MHS of level $m$ imply that 
$W_{m-1} H_\Q \cap F^{m/2} H_\C = 0$.
\end{proof}

Given an MHS $(H, N, F^\bu)$, the {\em primitive subspace} of $\Gr_{m+l}^W$ is defined to be
$$
P \Gr_{m+l}^W := \ker(N^{l+1}: \Gr_{m+l}^W \to \Gr_{m-l-2}^W).
$$

\begin{defn} For a non-negative integer $m$, a {\em polarized mixed Hodge structure} (PMHS, for short) {\em of level m} consists of the data of an MHS $H = (H,
  N, F^\bu)$  of level $m$ along with a non-degenerate, $(-1)^m$-symmetric $\Q$-bilinear form 
$$
S: H_\Q \otimes_\Q H_\Q \to \Q.
$$
The data $(H, N, F^\bu, S)$ are required to satisfy the following additional conditions: 
\begin{itemize}
\item $S(Na,b) + S(a,Nb) = 0$ for all $a,b \in H_\Q$, and
\item $S(F^p, F^q) = 0$ for all pairs of integers $p$ and $q$ satisfying $p+q = m+1$.
\end{itemize}
Using the first of these two properties, we may define, for each $l \geq 0$,  an induced pairing $S_l$ on the primitive subspace $P \Gr_{m+l}^W$ by 
$$
S_l(a, b) = S(\tilde{a}, N^l \tilde{b}), \text{ where $\tilde{a}, \tilde{b} \in W_{m+l}$ are representatives of $a,b$}.
$$
We extend $S_l$ to a  sesquilinear complex pairing on $P \Gr_{m+l}^W \otimes_\Q \C$ by
$$
S_l(a \otimes z, b \otimes w) = S_l(a,b) z \overline{w}.
$$
We also require that
\begin{itemize}
\item for each $l \geq 0$, we have $S_l(F^p P \Gr_{m+l}, F^{q} P \Gr_{m+l}) = 0$ for all $p$ and $q$ satisyfing $p+q = m+l +1$, and
\item for each $l \geq 0$ and all $p$,  we have
$\sqrt{-1}^{2p-m-l} S_l(a, \overline{a}) > 0$ whenever \newline $a \in F^p P \Gr_{m+l}^W \cap \overline{F^{m+l-p} P \Gr_{m+l}^W}$  and $a \ne 0$.
\end{itemize}
A morphism of PMHS's of level $m$, 
from $H = (H, N, F^\bu, S)$ to $H' = (H', N', F^\bu, S')$,  is a morphism $g$ of MHS's of level $m$ such that $S(a,b) = S'(g(a), g(b))$ for 
all $a,b \in H_\Q$. 
\end{defn}

\begin{ex}
Suppose $m$ is an even integer and $A$ is a finitely generated abelian group equipped with  the trivial MHS of level $m$. An extension of this data to a PMHS of level $m$ on $A$ consists of a
positive definite symmetric bilinear form $S$ on $A_\Q$. Indeed, the first three axioms hold trivially. The last one only has content when $l = 0$ and $p =
\frac{m}{2}$, in which case (since $\Gr_m A_\Q = A_\Q$) it  asserts that $S(a, \overline{a}) > 0$ for all $a \in A_\C$, where $S$ is extended 
sesquilinearly to $A_\C$ as before.
\end{ex}

\begin{lem} \label{lem94}
Let $m$ be an even integer, and suppose $H$ is a PMHS of level $m$. If $A$ is an abelian group, and $g: A \to H_\Q$ is a homomorphism of groups with image
contained in $\ker(N) \cap F^{\frac{m}{2}} H_\C$, then the induced pairing $S_A$ on $A_\Q$ defined by
$$
S_A(a,b) =  S(g(a), g(b))
$$
is a positive semi-definite symmetric bilinear pairing with kernel equal to $\ker(g_\Q)$. 
That is, $S_A(a,-) \equiv 0$ if and only if $a \in \ker(g_\Q)$, and the induced
pairing on $A_\Q/\ker(g_\Q)$ is positive definite.
\end{lem}

\begin{proof} It suffices to consider the case $A = \ker(N_{H_\Q}: H_\Q \to H_\Q) \cap F^{\frac{m}{2}} H_\C$ and $g$ is the identity map. 
That is, we just need to prove that the restriction of $S$ to this subspace $A$ is  positive
  definite (it is symmetric since $m$ is even). Given $\a \in A$, since $\ker(N_\Q) \subseteq W_m H_\Q$ and $\a \in F^{\frac{m}{2}}$,
we have an induced class 
$$
a := \Gr^W_m(\a) \in
  F^{\frac{m}{2}} \Gr^W_m.
$$ 
Since $\a \in H_\Q$, we have $\overline{a} =a$ and so 
$$
a \in F^{\frac{m}{2}} \Gr^W_m \cap \overline{F^{\frac{m}{2}} \Gr^W_m}.
$$
The axioms for $S$ (with $p = \frac{m}{2}$ and $l = 0$) give
$$
S_0(a,a) \geq 0, \text{ and if $S_0(a,a) = 0$, then $a = 0$,}
$$
where $S_0(a,a) = S(\a, \a)$. 
This proves that the restriction of $S$ to $A$ is positive semi-definite with kernel equal to the kernel of the canonical map $A \to
\Gr_m^W(H_\Q)$. But this map is injective by Lemma \ref{lem815}.
\end{proof}

\bibliographystyle{amsalpha}
\bibliography{Bibliography}
\end{document}